\documentclass[11pt,leqno]{article}
\usepackage{amsthm,amsfonts,amssymb,epsfig,graphics,amsmath,eufrak}
\usepackage[latin1]{inputenc}\relax

\def\tx{\tilde x}

\newcommand{\bU}{{\bar{U}}}
\newcommand{\sgn}{{\text{\rm sgn}}}

\newcommand{\Trace}{{\text{\rm Trace}}}

\newcommand{\R}{\Re}
\newcommand{\const}{\text{\rm constant}}

\newcommand{\RR}{{\mathbb R}}

\newcommand{\CC}{{\mathbb C}}

\newcommand{\EE }{{\mathbb E}}

\newcommand{\FF}{{\mathbb F}}
\newcommand{\GG}{{\mathbb G}}

\newcommand\cA{{\cal  A}}

\newcommand\cD{{\cal  D}}

\newcommand\cR{{\cal  R}}

\newcommand\cE{{\cal  E}}

\newcommand\cM{{\mathcal M}}

\newcommand\cS{{\mathcal S}}

\def\eps{\varepsilon }

\def\const{\text{\rm constant}}

\def\D{\partial }
\newcommand\adots{\mathinner{\mkern2mu\raise1pt\hbox{.}
\mkern3mu\raise4pt\hbox{.}\mkern1mu\raise7pt\hbox{.}}}

\renewcommand{\div}{{\rm div}}

\newcommand{\Span}{\text{\rm Span\ }}

\newtheorem{theo}{Theorem}[section]
\newtheorem{prop}[theo]{Proposition}
\newtheorem{cor}[theo]{Corollary}
\newtheorem{lem}[theo]{Lemma}
\newtheorem{defi}[theo]{Definition}

\newtheorem{rem}[theo]{Remark}
\newtheorem{rems}[theo]{Remarks}
\newtheorem{exams}[theo]{Examples}

\numberwithin{equation}{section}

\begin{document}

\title{\bf Stability of noncharacteristic boundary layers in
the standing-shock limit}

\author{\sc \small 
Kevin Zumbrun\thanks{Indiana University, Bloomington, IN 47405;
kzumbrun@indiana.edu:
Research of K.Z. was partially supported
under NSF grants number DMS-0070765 and DMS-0300487.  }}

\maketitle

\begin{abstract}
We investigate one- and multi-dimensional
stability of noncharacteristic boundary
layers in the limit approaching a standing planar shock wave
$\bar U(x_1)$, $x_1>0$,
obtaining necessary conditions of (i) weak stability of the
limiting shock, (ii) weak stability of the constant layer
$u\equiv U_-:=\lim_{z\to -\infty} \bar U(z)$, and
(iii) nonnegativity of a modified Lopatinski determinant
similar to that of the inviscid shock case.
For Lax $1$-shocks, we obtain equally simple sufficient conditions;
for $p$-shocks, $p>1$, the situation appears to be more complicated.
Using these results, we determine stability of certain isentropic and
full gas dynamical boundary-layers, generalizing earlier work
of Serre--Zumbrun and Costanzino--Humphreys--Nguyen--Zumbrun.
\end{abstract}

\tableofcontents


\section{Introduction}\label{intro}
Consider a boundary layer, or stationary solution,
\begin{equation}\label{profile}
\tilde U=\bU(x_1), \quad \lim_{z\to +\infty} \bU(z)=U_+, \quad
\bU(0)= U_0
\end{equation}
of a hyperbolic--parabolic system of conservation laws on the quarter-space
\begin{equation}\label{sys}
\tilde U_t +  \sum_jF_j(\tilde U)_{x_j} = \sum_{jk}(B_{jk}(\tilde
U)\tilde U_{x_k})_{x_j}, \quad x\in \mathbb{R}^{d}_+ =
\{x_1>0\},\quad t>0,
\end{equation}
$\tilde U,F_j\in \mathbb{R}^n$, $B_{jk}\in\mathbb{R}^{n \times n}$,
$$\tilde U = \begin{pmatrix}\tilde u\\
\tilde v\end{pmatrix}, \quad B=\begin{pmatrix}0 & 0 \\
b_{jk}^1 & b_{jk}^2\end{pmatrix},
\quad \tilde u\in \RR^{n-r}, \, \tilde v\in \RR^{r}, 
$$
with initial data $\tilde U(x,0)=\tilde U_0(x)$ and 
boundary conditions as specified in \eqref{BC} below,
that is noncharacteristic both in the hyperbolic sense 
$$
 \det dF^1(\bar U_+)\ne 0
$$
and with respect to the original (partially) parabolic problem
as described in (H1)--(H3) below.
Such layers occur, for instance, in gas- and magnetohydrodynamics with
inflow or outflow boundary conditions, for example in flow around
an airfoil with micro-suction or blowing; 
see \cite{GMWZ5,YZ,NZ1,NZ2} for further discussion.

As for any gas-dynamical flow, an important question is {\it stability}
of these solutions under perturbation of the initial or boundary data.
This question has been investigated in \cite{GR,MZ1,GMWZ5,GMWZ6,YZ,NZ1,NZ2}
for arbitrary-amplitude boundary-layers using Evans function techniques,
with the result that under quite general circumstances (see model
assumptions below) linearized and nonlinear stability reduce to a
generalized spectral stability condition phrased in terms of the {\it Evans 
function}, a Wronskian associated with the family of eigenvalue
ODE obtained by Fourier transform in the transverse directions
$\tilde x:=(x_2,\dots, x_d)$.
See also the small-amplitude results of
\cite{GG,R3,MN,KNZ,KK} obtained by energy methods.

The Evans function is readily evaluable numerically; see, e.g.,
\cite{CHNZ,HLyZ1,HLyZ2}.
As pointed out in \cite{SZ,CHNZ,GMWZ5}, it is also evaluable analytically
in certain interesting asymptotic limits.
For example, it was shown in \cite{GMWZ5} that the Evans function
converges in the small-amplitude limit as $\bar U$
approaches the constant layer $U\equiv U_+$ to the Evans function
of the constant layer, uniformly on compact sets of frequencies
$\Re \lambda \ge 0$,
and, as a consequence, {\it stability of small-amplitude layers is
determined by stability of the limiting constant layer}.
This result was used in turn to show that noncharacteristic boundary
layers of general symmetric--dissipative systems (defined below) are
spectrally stable in the small-amplitude limit.

A different asymptotic limit considered for special cases
in \cite{SZ,CHNZ} is the
{\it standing shock} limit $X\to +\infty$ in the case 
\begin{equation}\label{shockpiece}
\bar U^X(x)=\hat U(x_1-X),
\quad
\lim_{z\to \pm \infty}\hat U(z)=U_\pm
\end{equation}
that $\bar U$ is the restriction to $x_1>0$ of a standing
shock solution $\hat U(\cdot -X)$.
It is natural to guess that there might be some relation 
between boundary-layer stability in this limit and stability
of the limiting shock wave, and indeed it was shown in \cite{CHNZ}
for the case of isentropic ideal gas dynamics that the boundary
layer Evans function, suitably normalized, converges in the standing-shock
limit to the Evans function of the limiting shock wave 
on compact subsets of frequencies $\Re \lambda \ge 0$,
in complete analogy with the small-amplitude limit.
On the other hand, it was shown in \cite{SZ} for the case of full
(nonisentropic) ideal gas dynamics that boundary layers can in some
parameter regimes be {unstable} in the standing-shock limit
despite stability (see \cite{HLyZ1}) of the limiting shock wave.

In the present paper, we revisit the standing-shock limit in the
general case, obtaining a result subsuming and illuminating
these previous ones. 
Moreover, we carry out our investigations in multi-dimensions, whereas
the analyses of \cite{SZ,CHNZ} were specific to the one-dimensional
case.
Specifically, we show that the boundary-layer Evans function,
suitably normalized,
converges in the standing-shock limit, uniformly on compact subsets
of frequencies $\eta\in \RR^{d-1}$, $\Re \lambda \ge 0$, to the {product}
of the Evans function of the limiting shock wave and
the Evans function of the constant layer $U\equiv U_-$
at the left endstate of the shock.
For symmetric--dissipative systems, this implies by stability
of constant layers that the Evans function can be further renormalized
so as to converge simply to the Evans function of the limiting shock,
similarly as was shown in \cite{CHNZ} in the special one-dimensional
isentropic ideal gas case.

A consequence is that stability of both the limiting shock and
the constant layer $U\equiv U_-$ are {\it necessary conditions}
for stability of boundary layers in the standing-shock limit.
On the other hand, these are not sufficient, as even stable shock
waves have an (one-dimensional) eigenvalue at $\eta=0$, $\lambda=0$ 
due to translation-invariance, whereas stable boundary layers do not.
A further necessary condition, therefore, is nonnegativity of the
{\it stability index} (defined below) counting parity of the number
of (one-dimensional) unstable roots $\eta=0$, $\Re \lambda > 0$,
indicating that the zero eigenvalue of the limiting shock does not
perturb into the positive half-plane.
Negativity of the stability index in the standing-shock limit, 
indicating an odd number of unstable eigenvalues, was what was
shown in \cite{SZ} in order to obtain one-dimensional instability.

Here, we develop these ideas substantially further, determining for 
Lax $1$-shocks a simple and general stability determinant 
$\hat \Delta(\tilde \xi, \lambda, \eta)$ extending
the Lopatinski determinant $\Delta(\tilde \xi, \lambda)$ 
of the inviscid shock case, $\tilde \xi\in \RR$, $\eta\in \RR$,
$\lambda \in \CC$, with $\Re \lambda$, $\eta\ge 0$,
for which nonvanishing on the strictly positive half-space
$\Re\lambda$, $\eta >0$ is necessary and
nonvanishing on the nonnegative half-space $\Re \lambda$, $\eta \ge 0$
together with stability of the limiting shock $\hat U$ and the
constant layer $U\equiv U_-$
is sufficient for stability in the 
standing-shock limit, in one- and multi-dimensions.
We then use this condition to investigate stability in various interesting
situations.
For $p$-shocks, $p>1$, the situation is considerably
more tricky, apparently involving a complicated double limit.

\subsection{Equations and assumptions}

Consider a family $\bar U^X(x)$ of boundary-layers \eqref{shockpiece}
of \eqref{sys} consisting of translations of a standing shock solution
$\hat U$.
Following \cite{GMWZ5,GMWZ6}, we assume that the conservation law \eqref{sys}
can be rewritten in nonconservative form,
after an invertible change of variables $\tilde U\to \tilde W$, as a
quasilinear hyperbolic--parabolic system
\begin{equation}\label{symmsys}
 \tilde A_0(\tilde W)u_t  + \sum_{j=1}^d   \tilde A_j(\tilde W) \D_{j} (\tilde W)    -
 \eps \sum_{j,k= 1}^d \D_{j} \big( \tilde B_{jk}(\tilde W) \D_{k} \tilde W \big) = 0,
\end{equation}
$\tilde A_0$ invertible, with block structure
\begin{equation}
  \label{struc1}
  \tilde A_0(\tilde W)  =  \begin{pmatrix}\tilde A_0^{11}&0 \\\tilde A_0^{21}&\tilde A_0 ^{22}\end{pmatrix},
  \quad
  \tilde B_{jk} (\tilde W) =\begin{pmatrix}0&0 \\0 &\tilde B_{jk}^{22}\end{pmatrix},
  \end{equation}
a corresponding splitting
   $\tilde W = (\tilde w^1, \tilde w^2) \in \RR^{n-r} \times \RR^{r}$,
and decoupled boundary conditions
 \begin{equation}
\label{BC}
\left\{\begin{array}{l}
\Upsilon_1   (\tilde w^1 )_{ | x \in \partial \Omega  }   = g_1(t,x) , \\
\Upsilon_2   (\tilde w^2 )_{ | x \in \partial \Omega }   = g_2(t,x) , \\
\Upsilon_3  (\tilde w,   \D_{x_1} \tilde w^2,   \D_{\tilde x} \tilde w^2 )_{ | x \in \partial\Omega
} = 0 ,
\end{array}\right.
.
\end{equation}
where $\Upsilon_3  (\tilde w, \D_{x_1} \tilde w^2,
\D_{\tilde x} \tilde w^2 ) =
  K_1  \D_{x_1} \tilde w^2 
+  \sum_{j=2}^d  K_j (\tilde w)  D_{x_j} \tilde w^2 $,
$K_1\equiv \const$,
$\dim \Upsilon_1=n-r$ in the inflow case and $0$ in the outflow case
(defined in (H2) just below), and 
$$
\dim \Upsilon_2+\dim \Upsilon_3=r.
$$

We make the following technical hypotheses following 
\cite{Z1,Z3,GMWZ5}.
\medskip

(H0) $F^j, B^{jk}, \tilde A^0, \tilde A^j, \tilde B^{jk}, \tilde
W(\cdot), \tilde g(\cdot,\cdot) \in C^{s}$, $s\ge 2$.
\medskip

(H1)  The eigenvalues of $\sum_j (\tilde A_0^{11})^{-1} \tilde A_j^{11}\xi_j$ are real and semisimple
for all $\xi\ne 0$ in $\RR^d$.
\medskip

(H2) The eigenvalues
of $(\tilde A_0^{11})^{-1}\tilde A_1^{11}$ are either strictly positive or strictly
negative, that is, either $\sigma 
(\tilde A_0^{11})^{-1}\tilde A_1^{11}
\ge \theta_1>0$ 
(inflow case)
or $\sigma 
(\tilde A_0^{11})^{-1}\tilde A_1^{11}
\le -\theta_1<0$ (outflow case). 
\medskip

(H3) $ \Re\sigma
\sum_{jk} b_2^{jk}\xi_j\xi_k \ge
\theta |\xi|^2>0$ for all $\xi \in \RR^n\backslash \{0\}$.

\medskip

(H4) The eigenvalues of $\sum_j dF^j_\pm \xi_j$ 
are real, semisimple, and have constant multiplicity 
with respect to $\xi\in \RR^d$, $\xi\ne 0$.

(H5) The eigenvalues of $dF^1(U_\pm)$ are nonzero.

\medskip
(H6)
$ \R\sigma \left( \sum_j i\xi_j dF^j(U_\pm) - \sum_{j,k}
\xi_j\xi_kB^{jk}(U_\pm)\right) \leq -\theta \frac{|\xi|^2}{1+|\xi|^2} 
$ 
for all $\xi \in\RR^d$, some $\theta>0$.

%

\begin{defi}
\label{defsymm} The system \eqref{sys}, \eqref{symmsys} is symmetric
dissipative at $U_\pm$ if in a neighborhood of $U_\pm$ 
there exists a real matrix $S(\tilde U)$ depending
smoothly on $\tilde U$ such that for
such that for all $\xi \in \RR^d \backslash\{0\}$
$S(\tilde U) \tilde A_0(\tilde U) $ is symmetric definite positive and
block-diagonal,
$S(\tilde U)\sum_j  \tilde A_j (\tilde U) \xi_j $ is symmetric, and 
 $\Re  S(\tilde U) \sum  \tilde B_{jk}(\tilde U) \xi_j\xi_k $ is nonnegative with kernel of dimension $n-r$.
\end{defi}

\textbf{Alternative Hypothesis H4$'.$ }
For systems that are symmetric dissipative at $U_\pm$, we may relax (H4) to: 

(H4') About each $\xi\in \RR^d\setminus \{0\}$,
the eigenvalues of $\sum_j dF^j_\pm \xi_j$ 
(necessarily real and semisimple, by symmetrizability)
are either of constant multiplicity 
or else are totally nonglancing in the sense of \cite{GMWZ6}, Definition 4.3.

\begin{defi}\label{symmhp}
The system \eqref{sys}, \eqref{symmsys} is symmetric
hyperbolic--parabolic if there exists a real matrix $S(\tilde U)$ depending
smoothly on $\tilde U$ such that for
all $\xi \in \RR^d \backslash\{0\}$
the matrix $S (\tilde U) \tilde A_0(\tilde U) $ is symmetric positive definite
and block-diagonal,
$(S(\tilde U)\sum_j  \tilde A_j( \tilde U) \xi_j)^{11} $ is symmetric, 
and the symmetric matrix
 $\Re  S(\tilde U) \sum  \tilde B_{jk}(\tilde U) \xi_j\xi_k $ is nonnegative with kernel of dimension $n-r$.
\end{defi}

\begin{exams}\label{basicexams}
 \textup{Hypotheses (H1)--(H6) are satisfied for standing shocks
of the compressible Navier--Stokes equations with van der Waal 
equation of state, yielding boundary layers for which
the normal velocity of the fluid is nonvanishing
at $U_0$.
This corresponds to the situation of a porous
boundary through which fluid is pumped in or out, in contrast to the
characteristic, no-flux boundary conditions encountered at a solid
material interface for which normal velocity is set to zero.
See \cite{YZ,GMWZ5,NZ1,NZ2} for further discussion of this situation and 
applications to aerodynamics.}

\textup{Hypotheses (H1)-(H6) with  (H4) replaced by (H4')  are
satisfied  for extreme (i.e., $1$- or $n$-family) standing Lax shocks 
of the viscous MHD equations 
with van der Waal equation of state, with similar velocity restrictions 
on the plasma at $U_0$, but fail for intermediate shocks. 
Hypotheses (H1)--(H6) are generically satisfied for viscous MHD
in dimension one, but fail always for viscous MHD in dimensions greater
than or equal to two; see \cite{MZ2,GMWZ5,GMWZ6} for further discussion.
%
Both gas dynamics and MHD equations with van der Waals equation
of state are symmetric hyperbolic--parabolic systems that are
symmetric dissipative at $U_\pm$ for standing shocks connecting 
thermodynamically stable endstates \cite{Z3,GMWZ4}.}
\end{exams}

Finally, regarding the standing shock $\hat U$, 
ordering the eigenvalues of $dF(U_\pm)$ as
$$
a_1^\pm < a_2^\pm \dots < a_n^\pm,
$$
we assume:
\medskip

(H7) Profile $\hat U$ is a transversal viscous Lax $p$-shock, i.e.,
\begin{equation}
  a_{p-1}^-<0<a_p^-, \quad a_{p}^+<0<a_{p+1}^+
  \label{Lax}
\end{equation}
and $\hat U$ is a transversal connection of the standing 
wave ODE with boundary conditions $\tilde U_\pm$ (see \cite{MaZ3} for a 
detailed discussion of the standing wave ODE).  
\medskip

The eigenvalues $a_j^\pm$ correspond to characteristic speeds at $U_\pm$ 
of the associated one-dimensional inviscid system $U_t+F_1(U)_{x_1}=0$,
with $a_p$ the principal characteristic speed associated with the shock.

\subsection{The Evans condition}\label{evanscond}


The linearized eigenvalue equations of
\eqref{sys}, \eqref{BC} about $\bar U$ are
\begin{equation}\label{lind}
 \lambda U = LU :=
\sum_{j,k}( B_{jk} U_{x_k})_{x_j} - \sum_{j}( A_{j} U)_{x_j},
\end{equation}
with homogeneous boundary conditions
\begin{equation}\label{lBC}
\Upsilon'  (   W,   \D_{\tx}w^2, \D_{x_1}   w^2 ) _{|x_1 = 0}
 = 0
\end{equation}
expressed in linearized $\tilde W$-coordinates 
$W:=(\partial \tilde W/\partial \tilde U)(\bar U)U$,
where $W$ and $U$ denote perturbations of $\tilde W$ and $\tilde U$.

Taking the Fourier transform in $\tilde x:=(x_2,\dots,x_d)$, we
obtain a family of eigenvalue ODE
\begin{equation}\label{eigensys}\begin{aligned}
\lambda U=L_{\tilde \xi }U:=\overbrace{(B_{11}U')'-(A_1U)'}^{L_0U} - &i\sum_{j\not=1}A_j\xi_jU +
i\sum_{j\not=1}B_{j1}\xi_jU' \\&+i\sum_{k\not=1}(B_{1k}\xi_kU)' -
\sum_{j,k\not=1}B_{jk}\xi_j\xi_k U
\end{aligned}\end{equation}
with boundary conditions
\begin{equation}\label{eBC}
\Upsilon'  (   W,  i \tilde \xi   w^2, \D_{x_1}   w^2 ) _{|x_1 = 0}
 = 0.
\end{equation}

\subsubsection{The boundary-layer Evans function}\label{BLev}
A necessary condition for linearized stability is weak spectral
stability, defined as nonexistence of unstable spectra $\Re \lambda
>0$ of the linearized operator $L$ about the wave. As described in
Section \ref{const}, this is equivalent to nonvanishing for
all $\tilde \xi\in \RR^{d-1}$, $\Re \lambda>0$ of the {\it Evans
function}
$$
D(\tilde \xi, \lambda)
$$
a Wronskian associated with \eqref{eigensys} with columns consisting
of bases of the subspace of solutions decaying as $x_1\to +\infty$
and the subspace of solutions satisfying the boundary condition
\eqref{eBC}.
Under our hypotheses, the Evans function may be defined to be $C^\infty$ 
away from the origin on $\Re \lambda \ge 0$
with continous limits 
(typically depending on direction) at $(0,0)$ along rays through
the origin; see \cite{GMWZ5,GMWZ6}.

\begin{defi}\label{strongspectral}
\textup{ We define {\it strong spectral}, or {\it uniform
Evans} stability  as
\begin{equation}\tag{D}
|D(\tilde \xi, \lambda)|\ge \theta(C)>0
\end{equation}
for $(\tilde \xi, \lambda)$ on bounded subsets $C\subset \{\tilde
\xi\in \RR^{d-1}, \, \Re \lambda \ge 0\}\setminus\{0\}$. }
\end{defi}

\begin{rem}\label{stabrem}
\textup{
Under assumptions (H0)--(H6), uniform Evans stability implies
linearized and nonlinear stability in both the long time and 
small viscosity limits of general noncharacteristic boundary layers
(not necessarily associated with standing shocks) 
of symmetric hyperbolic--parabolic systems that are symmetric--dissipative
at $U_+$ with Dirichlet boundary conditions
$\dim\Upsilon_3=0$;
see \cite{GMWZ5,GMWZ6,NZ1,NZ2,N2}.
For more general systems and boundary conditions, (D) augmented with a rescaled high-frequency condition has been shown in \cite{GMWZ5,GMWZ6} to imply
stability in the small viscosity limit.
}
\end{rem}


\subsubsection{The shock Evans function}\label{shockev}
Likewise, a necessary condition for linearized stability of the
shock wave $\hat U$ is weak spectral
stability, defined as nonexistence of unstable spectra $\Re \lambda
>0$ of the linearized operator $L$ about the wave, 
or nonvanishing for
all $\tilde \xi\in \RR^{d-1}$, $\Re \lambda>0$ of the {\it shock Evans
function}
$$
\cD(\tilde \xi, \lambda)
$$
a Wronskian associated with \eqref{eigensys} with columns consisting
of bases of the subspace of solutions decaying as $x_1\to +\infty$
and the subspace of solutions decaying as $x_1\to -\infty$.
Under our hypotheses, the shock Evans function may be defined to be $C^\infty$ 
away from the origin on $\Re \lambda \ge 0$ and $C^0$ at the origin,
with first directional derivatives
(typically depending on direction) at $(0,0)$ along rays through
the origin; see \cite{Z3,GMWZ4,GMWZ6}. 

\begin{defi}\label{sstrongspectral}
\textup{ Uniform
Evans stability of a standing shock  is defined as
\begin{equation}\tag{$\cD$}
|\cD(\tilde \xi, \lambda)|\ge \theta(C)>0
\end{equation}
for $(\tilde \xi, \lambda)$ on bounded subsets $C\subset \{\tilde
\xi\in \RR^{d-1}, \, \Re \lambda \ge 0\}\setminus\{0\}$. }
\end{defi}
\begin{rem}\label{shockstabrem}
\textup{
Under assumptions (H0)--(H7), uniform Evans stability ($\cD$) implies
linearized and nonlinear stability in both the long time and 
small viscosity limits of standing shocks of symmetric hyperbolic--parabolic
systems that are symmetric--dissipative at $U_\pm$; see \cite{GMWZ4,Z3,N2}.
}
\end{rem}

\subsection{Main results}\label{results}

\subsubsection{Convergence}\label{subsec:conv}
Denote by 
$D_X(\tilde \xi, \lambda)$ the Evans function associated
with $\bar U^X$ and $D_-(\tilde \xi, \lambda)$ the Evans function associated
with the constant boundary-layer $U\equiv U_-$ at the lefthand
endstate $U_-$ of $\hat U$.
Then, our first main result is as follows.

\begin{theo}\label{conv}
Under assumptions (H0)--(H7),
there exists a continuous nonvanishing function $\beta(\tilde \xi, \lambda,X)$
such that
\begin{equation}\label{eq:conv}
D_X(\tilde \xi, \lambda)\to 
\beta(\tilde \xi, \lambda,X) D_-(\tilde \xi, \lambda)\cD(\tilde \xi, \lambda)
\end{equation}
as $X\to \infty$, uniformly on compact subsets of $\{\tilde \xi\in \RR^{d-1},
\, \Re \lambda \ge 0\}$.
\end{theo}

\begin{cor}
Under (H0)--(H7), weak spectral stability of both the constant boundary-layer
$U\equiv U_-$ and the standing shock $\hat U$ are necessary conditions
for stability of $\bar U^X$ in the standing-shock limit $X\to \infty$.
\end{cor}

\begin{rems}\label{hfrems}
\textup{
1. Under the stronger definition of uniform Evans stability defined
in \cite{GMWZ5,GMWZ6} involving also a rescaled high-frequency condition,
and assuming (H0)--(H7),
uniform Evans stability of the constant boundary-layer
$U\equiv U_-$ and the standing shock $\hat U$ are also {\it sufficient} 
conditions
for spectral stability of $\bar U^X$ in the standing-shock limit $X\to \infty$
for frequencies uniformly bounded away from the origin $(\tilde \xi, \lambda)=(0,0)$.
For intermediate frequencies $R^{-1}\le |(\tilde \xi, \lambda)|\le R$,
this is an immediate consequence of Theorem \ref{conv}.
For high frequencies $|(\tilde \xi, \lambda)|\ge R$,
$R>0$ sufficiently large, it follows from the fact established in
Section 3.2 \cite{GMWZ5} that high-frequency stability
is equivalent to stability of the constant layer $U\equiv U_0$,
and the fact that $U^X_0\to U_-$ as $X\to \infty$.
That is, {\it assuming stability of $\hat U$ and $U\equiv U_-$,
unstable frequencies, if they occur, must converge to the origin
as $X\to \infty$,} with no additional assumptions on the system \eqref{sys}.
}

\textup{
2. It is shown in Corollary 1.29 \cite{GMWZ5} that constant boundary-layers
of symmetric--dissipative systems are uniformly Evans stable (under
the stronger definition of \cite{GMWZ5,GMWZ6}) for Dirichlet boundary
conditions. 
Thus, {\it if \eqref{sys}
is symmetric--dissipative at $U_-$ and the boundary conditions are
Dirichlet-type, then the constant layer $U\equiv U_-$
is stable}, with the implications above.
In the general case, stability or instability of the constant 
layer $U\equiv U_-$ may be determined by a linear algebraic computation, 
since the constant-coefficient eigenvalue ODE is explicitly soluble 
for each frequencies $(\tilde \xi$, $\lambda)$; see \cite{GMWZ5} for
further discussion.
}
\end{rems}

\subsubsection{One-dimensional stability}
Due to translational invariance, $\cD(0,0)=0$, and so 
we cannot conclude nonvanishing of $D_X$ near the origin 
as $X\to \infty$ from the convergence result \eqref{eq:conv}.
Indeed, it is possible that the zero of $\cD$ at the origin
may perturb into the unstable half-plane $\Re \lambda>0$
for boundary layers with $X$ large, yielding instability.
In the one-dimensional setting, this may be dectected by
the {\it stability index}
$$
\Gamma:=
\sgn \lim_{\lambda\to 0^+ \, \hbox{\rm real}} D(0,\lambda)
\times
\lim_{\lambda\to +\infty \, \hbox{\rm real}} \sgn D(0,\lambda),
$$
where $D$ is chosen with a standard normalization guaranteeing
that it is real for real $\lambda$ and $\tilde \xi=0$.
The stability index 
is well-defined by the properties that $D$ is continuous along
rays at the origin and nonvanishing for real $\lambda$ sufficiently
large; see \cite{GZ,SZ,Z3} for further discussion.
Negativity of $\Gamma$, by the Intermediate Value Theorem, implies existence
of a real positive root $D(0,\lambda_*)=0$, hence
{\it one-dimensional instability}.

Observing that stability of the constant layer $U\equiv U_-$
and the limiting shock $\hat U$
imply that $\sgn D_-(\lambda)$ and $\sgn \cD(\lambda)$ are
constant for real $\lambda\ge 0$,
we thus obtain by convergence, \eqref{eq:conv},
that a necessary condition for stability in the standing shock
limit $X\to +\infty$, assuming stability of the constant layer,
is nonnegativity of 
\begin{equation}\label{form}
\hat\Gamma:=
\lim_{X\to +\infty}
\sgn \lim_{\lambda\to 0^+ \, \hbox{\rm real}} 
\frac{D_X(0,\lambda)} {D_-(0,\lambda)}
\times
\lim_{\lambda\to 0^+ \, \hbox{\rm real}} \sgn \cD(0,\lambda),
\end{equation}
provided these limits exist.

A standard shock stability result \cite{GZ,ZS,Z3} is that,
with appropriate normalization,
\begin{equation}\label{shocknorm}
\begin{aligned}
\lim_{\lambda\to 0^+ \, \hbox{\rm real}} \sgn \cD(0,\lambda)&=
\sgn \partial_\lambda \cD(0,0)= \sgn \delta,\\
\delta&:= \sgn\det(R^-, R^+, [U]), 
\end{aligned}
\end{equation}
where $R^-$ and $R^+$ are matrix blocks whose columns span the
stable subspace of $dF_1(U_-)$ and the unstable subspace of $dF_1(U_+)$
and $[U]:=U_+-U_-$ denotes the jump across the shock.
The determinant $\det(R^-, R^+, [U])$ may be recognized
as the {\it Lopatinski determinant} of one-dimensional inviscid
theory, whose nonvanishing is equivalent to 
one-dimensional inviscid shock stability.

Our second main result asserts that the first limit in \eqref{form}
also exists, yielding a necessary stability condition of nonnegativity
of a certain Lopatinski-like determinant $\hat \delta$
relative to the sign of $\delta$. 

\begin{theo}\label{lim}
Assuming (H0)--(H7) and stability of the constant-layer $U\equiv U_-$
and the limiting shock $\hat U$,
$$
\begin{aligned}
\lim_{X\to +\infty} \sgn \lim_{\lambda\to 0^+ \, 
\hbox{\rm real}} 
\frac{D_X(0,\lambda)} {D_-(0,\lambda)} &= \sgn \hat \delta,\\
\hat \delta&:= \sgn \det (\hat R^-,R_+,\hat V)
\end{aligned}
$$
provided $\hat \delta \ne 0$
and $S:= \lim_{z\to -\infty}\frac{\hat U'}{|\hat U'|}(z)$ exists,
where $\hat R^-$ has the dimensions of $R^-$ and $\hat V$ is
a single column vector, hence
\begin{equation}\label{1stabcond}
\sgn \delta \hat \delta=
\sgn \det ( R^-,R_+,[U]) \det (\hat R^-,R_+,\hat V)\ge 0
\end{equation}
is necessary for one-dimensional stability in the standing-shock
shock limit.

For a Lax $1$-shock with Dirichlet boundary conditions,
\begin{equation}\label{1lop}
\delta= \det(R_+, [U]) \; \hbox{\rm and} \;
\hat \delta=\det(R_+, dF_1(U_-) S)
\end{equation}
so that \eqref{1stabcond} becomes
\begin{equation}\label{1Dstabcond}
\sgn \det ( R_+,[U]) \det (R_+,dF_1(U_-)S)\ge 0.
\end{equation}
\end{theo}


\begin{rems}\label{SZrmk}
\textup{1. In the (characteristic) limit $U_-\to U_+$ as the
amplitude of the background shock $\hat U$ goes to zero,
$S\sim [U] \sim r_1^\pm$, where $r_1$ is the eigenvector
of $A_1$ associated with the smallest eigenvalue $a_1$,
and $dF_1(U_-)S\sim A_1^-r_1^-=a_1^-[U]$, where $a_1^->0$
by the Lax shock conditions \eqref{Lax}.
Thus, $\sgn \det ( R_+,[U]) \det (R_+,dF_1(U_-)S)\sim \sgn \det ( R_+,[U])^2>0$,
consistent with stability.
}

\textup{
2. From the boundary-layer ODE, $dF_1(\hat U)\hat U'=(B_{11}(\hat U)\hat U')'$,
giving $A_1(x_1)\hat U'= B_{11}\hat U''$, where $'$ denotes $\partial_{x_1}$.
Thus, $A_1^-S=\alpha B_{11}^- S$, where $\alpha =\lim_{z\to -\infty}
(\hat U'/\hat U)(z)$ is necessarily real and positive if the limit
$S:= \lim_{z\to -\infty} (\hat U'/|\hat U'|)(z)$ exists.
It follows that we may replace \eqref{1Dstabcond} by the equivalent condition
\begin{equation}\label{SZversion}
\det(R_+, [U]) ) \det(R_+, B_{11}(U_-) S)\ge 0,
\end{equation}
which may be recognized as the necessary condition derived 
by a rather different argument in Section 4.1 of \cite{SZ}.
}
\end{rems}

Our third result states that for Lax $1$-shocks 
the necessary conditions we have derived 
are also essentially {\it sufficient} for one-dimensional 
stability in the standing shock limit.

\begin{theo}\label{1lim}
For a Lax $1$-shock with general boundary conditions \eqref{BC}, 
assuming (H0)--(H7), positivity of $\hat \Gamma$ 
together with stability of the constant layer $U\equiv U_-$
and the limiting shock $\hat U$ is {\rm sufficient} for stability
in the standing-shock limit $X\to +\infty$.
\end{theo}

\begin{rems}\label{suffeg}
\textup{
1. It was shown in \cite{HuZ} that
in the limit $U_-\to U_+$, the background shock $\hat U$ is stable.
For symmetric-dissipative systems with Dirichlet boundary conditions,
the constant layer $U\equiv U_-$ is stable.
By Remark \ref{SZrmk}.1, therefore, for Lax $1$-shocks of symmetric-dissipative
systems, with Dirichlet boundary conditions, layers $\hat U^X$ are
one-dimensionally
stable in the standing-shock limit $X\to +\infty$ for $U_+$ fixed and
shock amplitude $|U_+-U_-|$ sufficiently small.
}

\textup{
2.  By Remark \ref{SZrmk}.2, the necessary condition of \cite{SZ},
together with stability of the limiting shock $\hat U'$ and
the constant layer $U\equiv U_-$, is
also sufficient for one-dimensional stability.
For ideal gas dynamics, the numerical study of \cite{HLyZ1} indicates
one-dimensional stability of arbitrary shock waves $\hat U$ for
gas constant $\gamma$ within the physical range $1.2\le \gamma \le 3$
(the only values considered); likewise, $U\equiv U_-$ is stable by
symmetric--dissipativity of the compressible Navier--Stokes equations.
Thus, {\it for ideal gas dynamics with Dirichlet boundary conditions
and $1.2\le \gamma \le 3$, one-dimensional stability 
in the standing Lax $1$-shock limit
is completely decided by the simple algebraic condition \eqref{SZversion} 
of \cite{SZ}}.
}
\end{rems}

For Lax $p$-shocks, $p\ge 2$, the situation is more complicated,
apparently involving a tricky double limit.
In particular, the conditions of Theorem \ref{lim}
are only necessary and not sufficient for stability.

\subsubsection{Multi-dimensional stability}

We restrict now to the case of a Lax $1$-shock, for simplicity
taking pure Dirichlet boundary conditions, $\dim \Upsilon_3=0$.
In multi-dimensions, uniform inviscid stability of a Lax $1$-shock
is defined as
nonvanishing of the multi-dimensional Lopatinski determinant
\begin{equation}\label{mlop}
\Delta(\tilde \xi, \lambda):=
\det \Big( 
\cR_+(\tilde \xi, \lambda),
\lambda [U] + \sum_{j=2}^d i\xi_j [F_j(U)]
\Big),
\end{equation}
on the nonnegative unit sphere $\cS^+:=\{|(\tilde \xi, \lambda)|=1,
\, \Re \lambda \ge 0\}$,
where $\cR_+$ is a matrix blocks whose columns form a basis for
the unstable subspaces of 
\begin{equation}\label{cA}
\cA_+(\tilde \xi, \lambda)
:= \Big(\lambda I + \sum_{j=2}^d i\xi_j dF_j(U_+)\Big)dF_1(U_+)^{-1}
\end{equation}
and $[h(U)]:=h(U_+)-h(U_-)$ denotes jump in $h$ across the shock.
Define the related determinant
\begin{equation}\label{hatDelta}
\begin{aligned}
\hat \Delta(\tilde \xi, \lambda, \eta):=
\det \Big( \cR_+(\tilde \xi, \lambda) ,
\lambda [U] + \sum_{j=2}^d i\xi_j [F_j(U)] + \eta dF_1(U_-) S_- \Big),
\end{aligned}
\end{equation}
$\eta\in \RR$, where $S:=
\lim_{z\to -\infty}\frac{\hat U'}{|\hat U'|}(z)$.
Then, our fourth and fifth main results, giving necessary conditions
and sufficient conditions for multi-dimensional stability
analogous to those of Theorems \ref{lim} and \ref{1lim}, are as follows.

\begin{theo}\label{mnec}
For a Lax $1$-shock and Dirichlet boundary conditions,
assuming (H0)--(H7), 
a necessary condition for stability of $\bar U_X$
in the standing-shock limit $X\to +\infty$ is that
$\hat \Delta$ have no root that is simple with respect to $\lambda$
on the positive half-sphere 
$$
\hat \cS^+:=
\{|\tilde \xi, \lambda, \eta|=1, \, \Re \lambda > 0, \, \eta > 0\},
$$
in the sense that $\hat \Delta=0$ and $\partial_\lambda \hat \Delta\ne 0$.
\end{theo}

\begin{theo}\label{msuff}
For a Lax $1$-shock and Dirichlet boundary conditions,
assuming (H0)--(H7), 
sufficient conditions for stability of $\bar U_X$
in the standing-shock limit $X\to +\infty$ are nonvanishing
of $\hat \Delta$ on the nonnegative half-sphere $\hat \cS^+:=
\{|\tilde \xi, \lambda, \eta|=1, \, \Re \lambda \ge 0, \, \eta \ge 0\}$,
stability of the limiting shock $\hat U$, and stability
of the constant layer $U\equiv U_-$, 
\end{theo}

\begin{rem}
\textup{
In the one-dimensional case $\tilde \xi\equiv 0$, $\cR_\pm\equiv R_\pm=\const$,
so that $\partial_\lambda \hat \Delta\equiv \delta$.
It is readily seen that Theorems \ref{mnec} and \ref{msuff}
reduce in this context
to the restrictions of Theorems \ref{lim} and \ref{1lim} to
the case of a Lax $1$-shock and Dirichlet boundary conditions, 
provided that the limiting shock $\hat U$ is
one-dimensionally stable, so that $\delta \ne 0$ \cite{ZH}.
}
\end{rem}

\begin{rems}\label{mSZrmk}
\textup{
1. Multi-dimensional stability of shock waves in the limit
$U_-\to U_+$ has been established in \cite{FS} for symmetric
dissipative systems with strictly
parabolic (Laplacian) viscosity $\sum (B_{jk}U_{x_j})_{x_k}$
By the arguments of Remarks \ref{SZrmk}.1 and \ref{suffeg}.1,
therefore, boundary layers of such systems with Dirichlet
boundary conditions are multi-dimensionally 
stable in the standing Lax $1$-shock
limit for $U_+$ fixed and $|U_+-U_-|$ sufficiently small.
The argument of \cite{FS} appears likely to generalize
to the general symmetric dissipative case (see also
\cite{PZ}), which would extend the boundary layer result
also to the general case.
}

\textup{
2. For ideal gas dynamics, the numerical study of 
\cite{HLyZ2} indicates multi-dimensional stability 
of arbitrary shock waves for gas constants in the
physical range $1.2\le \gamma\le 3$.
By the arguments of Remarks \ref{SZrmk}.2 and \ref{suffeg}.2,
therefore, {\it for ideal gas dynamics with Dirichlet boundary conditions
and $1.2\le \gamma \le 3$,
multi-dimensional stability in the standing Lax $1$-shock limit
is completely decided by vanishing or nonvanishing
of the extended Lopatinski determinant $\hat \Delta$
defined in \eqref{hatDelta}: a simple, linear-algebraic condition.}
}
\end{rems}

\subsubsection{Verification}\label{ver}

Evidently, nonvanishing of $\hat \Delta$ on the nonnegative
half-sphere in $(\tilde \xi, \lambda,\eta)$ 
is equivalent to the condition that the image of
\begin{equation}\label{hateta}
\hat \eta(\tilde \xi, \lambda):=
- \Delta(\tilde \xi, \lambda)/
\det \Big( \cR_+(\tilde \xi, \lambda) , dF_1(U_-) S_- \Big)
\end{equation}
over the nonnegative half-sphere in $(\tilde \xi, \lambda)$ 
avoid the nonnegative real axis, a natural generalization of
the one-dimensional condition \eqref{1Dstabcond}.
Note that $\hat \eta$ is independent of the choice of $\cR_+$
and homogeneous degree one in $(\tilde \xi, \lambda)$.
This leads us to the following condition convenient for
numerical or analytic verification.

\begin{prop}\label{verprop}
Assuming the one-dimensional stability condition \eqref{1Dstabcond},
nonvanishing of $\hat \Delta$ on the nonnegative
half-sphere in $(\tilde \xi, \lambda,\eta)$ 
is equivalent to the condition that the image of $\hat \eta(1,i\tau)$
over $\tau\in \RR$, with $\hat \eta$ defined as in \eqref{hateta}, 
does not intersect the nonnegative real axis.
\end{prop}

\begin{proof}
By homogeneity of $\hat \Delta$, nonvanishing on the half-sphere is
equivalent to nonvanishing of $\hat \delta(0,\lambda, \eta)$, or
\eqref{1Dstabcond},
and nonvanishing of $\hat \delta(1, \lambda, \eta)$, or
the condition that $\hat \eta(1,\lambda)$ avoid the nonnegative real
axis for $\Re \lambda\ge 0$.
Recalling the standard fact that $\cR_+(1,\lambda)$, 
hence $\hat \eta(1,\lambda)$, may be chosen to be analytic in $\lambda$
for $\Re \lambda>0$ and continuous at $\Re \lambda=0$, we find by 
the argument principle applied to a sufficiently large semicircle about
the origin, bounded to the left by the imaginary axis,
and the fact that $\hat \eta(1, \lambda)$ by homogeneity/continuity
does not intersect the nonnegative real axis for $\Re \lambda \ge 0$
and $|\lambda|$ sufficiently large (by the assumed one-dimensional stability)
that the latter condition is
equivalent to nonintersection with the nonnegative real axis for 
as $\lambda$ traverses the imaginary axis.
\end{proof}

\subsection{Discusssion and open problems}\label{discussion}

The results of Theorems \ref{mnec} and \ref{msuff} illuminate and
greatly extend the earlier results of \cite{SZ} and \cite{CHNZ}
in the one-dimensional case.
In particular, we regard the derivation of
necessary and sufficient conditions for multi-dimensional stability
as a substantial advance.
Though our necessary and our sufficient conditions are slightly different,
the difference is sufficiently slight that it should not interfere in practice
with classification of physical stability regions.

We note that, besides its independent interest, the treatment of the
standing shock limit, as pointed out in \cite{CHNZ}, is important
in truncating the computational domain for global stability analyses.

On the other hand, we have restricted here mainly to the simplest
case of a Lax $1$-shock, which corresponds to the case of {\it inflow}
boundary conditions.  It would be very interesting to obtain corresponding
conditions also in the case of outflow boundary conditions, for example,
a Lax $n$-shock.
Likewise, it would be very interesting to carry out computations analogous
to those carried out for gas dynamics in Section \ref{sec:apps} also for the
equations of MHD.

\section{Construction of the Evans function}\label{const}

We begin by reviewing the construction of the Evans function
following \cite{Z3,GMWZ5,GMWZ6,NZ1,NZ2}.

\subsection{Expression as a first-order system}
We first observe that matrix
$$
\begin{pmatrix} A_1^{11}& A_1^{12}\\ 
B_{11}^{21}& B_{11}^{22}\end{pmatrix}
=
\begin{pmatrix} dF_1^{11}& dF_1^{12}\\ 
B_{11}^{21}& B_{11}^{22}\end{pmatrix}
$$ 
is full rank, by (H2)--(H3) together with 
block structure assumption \eqref{symmsys},
as can be seen most easily by working in $\tilde W$-coordinates;
see \cite{Z1,Z3,MaZ3}. 

As a consequence, the Fourier-transformed eigenvalue equations
\eqref{eigensys}, \eqref{eBC} may be written as a first-order system
\begin{equation}\label{firstorder}
\begin{aligned}
Z'&=\GG_X(\tilde \xi, \lambda, x_1)Z,
\\
&
M(\tilde \xi) Z= 0
\end{aligned}
\end{equation}
in the convenient coordinates
\begin{equation}\label{Z}
Z=
\begin{pmatrix} Z_1\\ Z_2\\ Z_3 \end{pmatrix}
:=
\begin{pmatrix}
B_{11}^{21}u+B_{11}^{22}v\\
B_{11}U'- A_1 U
\end{pmatrix}
=
\begin{pmatrix}
B_{11}^{21}u+B_{11}^{22}v\\
-A_{1}^{11}u- A_1^{12} v\\
B_{11}^{21}u'+B_{11}^{22}v'
-A_{1}^{21}u- A_1^{22} v\\
\\
\end{pmatrix};
\end{equation}
see \cite{Z3} for further discussion.

In the case of Dirichlet boundary conditions $\dim \Upsilon_3=0$,
the boundary operator $M$ is independent of $\tilde \xi$, with
\begin{equation}\label{dirkerin}
\ker M=\{Z:\, Z_1=0, \, Z_2=0\} 
\end{equation}
in the inflow case and
\begin{equation}\label{dirkerout}
\ker M=\{Z:\,  Z_1=0\} 
\end{equation}
in the outflow case.
In general, $\ker B(\tilde \xi)$ depends on $\tilde \xi$ in
a possibly complicated way.
By the underlying relation $\bar U_X(x_1)=\hat U(x_1-X)$, we have
\begin{equation}
\GG_X(\tilde \xi, \lambda, x_1)= \GG(\tilde \xi, \lambda, x_1-X).
\end{equation}

\subsection{Conjugation to constant-coefficients}
We next recall the following result established in \cite{Z3,GMWZ5},
a consequence of the conjugation lemma introduced in \cite{MZ1}
and the fact proved in \cite{MaZ3,Z3} 
that $\hat U$ under hypotheses (H0)--(H7) converges
exponentially to $U_\pm$ as $x_1\to \pm \infty$.
See, e.g., \cite{Z3} for details.

\begin{prop}\label{conjugation}
There exist matrix-valued functions 
$$
T_X^\pm(\tilde \xi, \lambda,x_1)= T^\pm(\tilde \xi, \lambda, x_1-X), 
$$
uniformly bounded with bounded inverse for $x_1\gtrless 0$,
locally analytic in $(\tilde \xi, \lambda)$, such that
\begin{equation}\label{exp}
|(T^\pm-Id)(x_1)|\le Ce^{-\theta |x_1|} \; \hbox{\rm for}\; x_1 \gtrless 0,
\end{equation}
$\theta>0$,
and $Z:=T_X^\pm X$ satisfies the constant-coefficient equation
$$
X'=\GG_X^\pm(\tilde \xi,\lambda)X \; \hbox{\rm for} \; x_1\gtrless X
$$
$\GG_X^\pm(\tilde \xi,\lambda):= \GG_X^\pm(\tilde \xi,\lambda,\pm \infty)$,
whenever $Z$ satisfies \eqref{firstorder}.
\end{prop}

\subsection{Definition of the Evans function}\label{evansdef}

Finally, we recall the following standard result, established in
increasing generality in \cite{SZ,Z1,Z3,GMWZ4,GMWZ6}.

\begin{prop} \label{cbases}
Under (H0)--(H7), there exist matrices $\EE^-$, $\EE^+$, and
$\EE^0$ whose columns form bases of the unstable subspace of
$\GG^-$, the stable subspace of $\GG^+$, and $\ker \cM(\tilde \xi)$
and matrices $\FF^-$ and $\FF^+$ whose columns form bases of the
the stable subspace of $\GG^-$ and 
the unstable subspace of $\GG^+$, $C^\infty$ on $\{\Re \lambda \ge 0\}
\setminus \{(0,0)\}$ and continuously extendable along rays through the
origin.
Moreover, $\EE^+$ and $\FF^+$ and $\EE^-$ and $\FF^-$ 
 are uniformly transverse on
compact subsets of $\{\Re \lambda \ge 0\}$ and
\begin{equation}\label{dims}
\dim \Span \EE^-= \dim \Span \EE^0= (n+r)-\dim \Span \EE^-.
\end{equation}
\end{prop}

\begin{rem}\label{analytic}
\textup{
In the one-dimensional setting, the subspaces of Proposition \ref{cbases}
may be defined as {\it globally analytic} functions on $\Re \lambda \ge -\eta$, 
$\eta>0$, using a standard construction of Kato \cite{Kat}; see, e.g., 
\cite{GZ,Z3,CHNZ}.
}
\end{rem}

With these preparations, we define the shock Evans function precisely as
the $(n+r)\times (n+r)$ Wronskian
\begin{equation}\label{psev}
\cD(\tilde \xi, \lambda):=
\det \Big( T^- \EE^-, T^+\EE^+\Big)|_{x_1=0}
\end{equation}
and the boundary-layer Evans function as
\begin{equation}\label{pblev}
D_X(\tilde \xi, \lambda):=
\det \Big( \EE^0, T_X^+\EE^+\Big)|_{x_1=0}=
\det \Big( \EE^0, T^+\EE^+\Big)|_{x_1=-X}.
\end{equation}
The Evans function for the constant-layer $U\equiv U_-$ is given by
$$
D_-(\tilde \xi, \lambda):=
\det \Big( \EE^0, \FF^-\Big),
$$
since in this case $\EE^+=\FF^-$, or equivalently by
\begin{equation}\label{pcblev}
D_-(\tilde \xi, \lambda):=
\det \Big( \tilde \EE^{-*}\EE^0\Big),
\end{equation}
where 
$$(\tilde \EE^-, \tilde \FF^-)= (\EE^-, \FF^-)^{-1*}$$
are dual bases to $(\EE^-, \FF^-)$
with respect to the standard complex inner product,
$M^*$ denoting adjoint, or conjugate transpose,
of a matrix $M$.

\subsection{Behavior near zero}\label{sub:behavior}

For later use, we record the following refinement of Proposition \ref{cbases}
(also established in \cite{SZ,Z1,Z3,GMWZ4,GMWZ6}),
from which we may determine the
behavior of $\cD$, $D_x$ as $(\tilde \xi, \lambda)\to (0,0)$.

\begin{prop}\label{Elim}
Under (H0)--(H7), the limiting subspaces
$$
\lim_{\rho \to 0^+}\Span T^+ e^{\GG^+x_1}\EE^+(\rho \tilde \xi_0, \rho\lambda_0)
\;
\hbox{\rm  and } \; 
\lim_{\rho \to 0^+}\Span T^- e^{\GG^-x_1}\EE^-(\rho \tilde \xi_0, \rho\lambda_0) 
$$
$\rho\in \RR$, $\Re \lambda_0 \ge 0$, 
are spanned by the direct sum of fast modes 
$$ Z_1=b_{11}\phi,
\quad 
(Z_2,Z_3)=0,
$$
with $\phi$ satisfying $B_{11}\phi'-A_1\phi=0$ 
and decaying at $+\infty$ [resp. $-\infty$] and slow modes 
$$
Z_1= * ,\quad (Z_2,Z_3)=\cR,
$$
with $\cR$ spanning the unstable [resp. stable] subspace of
$\cA_+$ [resp. $\cA_-$, with
$$
\cA_\pm(\tilde \xi, \lambda)
:= \Big(\lambda I + \sum_{j=2}^d i\xi_j dF_j(U_\pm)\Big)dF_1(U_\pm)^{-1},
$$
expressible alternatively as $Z=T^\pm X$ with
\begin{equation}\label{Xconst}
X_1=b_{11}(-A_1^\pm)^{-1}\cR,\quad (X_2,X_3)=\cR
\end{equation}
constant.
Symmetric decompositions hold for 
$T^+e^{\GG^+x_1}\FF^+$ and $T^-e^{\GG^-x_1}\FF^-$.
\end{prop}

\section{Basic convergence result}\label{sec:conv}

\begin{proof}
[Proof of Theorem \ref{conv} for $(\tilde \xi, \lambda)$
 bounded away from the origin]
Viewing $\cD$ and $D_X$ as Wronskians of solutions of the same 
ODE \eqref{firstorder}, we may rewrite \eqref{pblev} using Abel's Theorem
as
\begin{equation}\label{Ablev}
D_X(\tilde \xi, \lambda):=
e^{\int_0^{-X} \Trace \GG(\tilde \xi, \lambda, z)dz}
\det \Big( \cS^{-X\to 0}\EE^0, T^+\EE^+\Big)|_{x_1=0},
\end{equation}
where $\cS^{y\to x}$ denotes the solution operator of \eqref{firstorder}.

Next, expand
\begin{equation}\label{keycalc}
\begin{aligned}
\EE^0&= T^-(-X) T^-(-X)^{-1}\EE^0\\
&=
T^-\EE^0|_{x_1=-X} + O(e^{-\theta X})\\
&=
T^- \Pi_{\EE^-}\EE^0|_{x_1=-X} 
+T^- \Pi_{\FF^-}\EE^0|_{x_1=-X} 
+ O(e^{-\theta X}),\\
\end{aligned}
\end{equation}
where
\begin{equation}\label{proj}
\Pi_{\EE^-}= \EE^- \tilde \EE^{-*}
\; \hbox{\rm and } \;
\Pi_{\FF^-}= \FF^- \tilde \FF^{-*}
\end{equation}
denote the eigenprojections of $\GG^-$ onto
subspaces $\EE^-$ and $\FF^-$, noting that
\begin{equation}\label{projform}
T^- \Pi_{\EE^-}\EE^0 = (T^- \EE^-)(\tilde \EE^{-*}\EE^0).
\end{equation}

From \eqref{projform}, we obtain
\begin{equation}\label{piece}
\begin{aligned}
\det &\Big( \cS^{-X\to 0}T^-\Pi_{\EE^-}\EE^0, T^+\EE^+\Big)|_{x_1=0}\\
&= \det \Big( \cS^{-X\to 0}T^-\EE^-, T^+\EE^+\Big)|_{x_1=0}
\det (\tilde \EE^{-*}\EE^0)\\
&=
\det \Big( \cS^{-X\to 0}T^-\EE^-, T^+\EE^+\Big)|_{x_1=0} D_-.
\end{aligned}
\end{equation}
Next, expanding
$$
\cS^{y\to x}T^-(y)=      T^-(x)e^{\GG^-(x-y)}
$$
and noting that
\begin{equation}\label{commutator}
e^{\GG^-(x-y)}\EE^-=
e^{\Pi_{\EE^-}\GG^-(x-y)}\Pi_{\EE^-}\EE^-=
\EE^-e^{\tilde \EE^{-*}\GG^-  \EE^-(x-y)},
\end{equation}
where $\Trace (\tilde \EE^{-*}\GG^-  \EE^-)= 
\Trace (\tilde \EE^{-*} \EE^- \GG^-)= \Trace \Pi_{\EE^-}\GG^-$,
we find that
\begin{equation}\label{2piece}
\begin{aligned}
\det &\Big( \cS^{-X\to 0}T^-\EE^-, T^+\EE^+\Big)|_{x_1=0}\\
&\quad =
e^{\Trace(\Pi_{\EE^-}\GG^-)X}
\det &\Big( T^-\EE^-, T^+\EE^+\Big)|_{x_1=0}\\
&\quad =
e^{\Trace(\Pi_{\EE^-}\GG^-)X}\cD,
\end{aligned}
\end{equation}
where, since $\EE^-$ is the unstable subspace of $\GG^-$,
$$
e^{\Trace(\Pi_{\EE^-}\GG^-)X}
$$
is uniformly exponentially growing in $X$ for $\Re \lambda \ge 0$.
Indeed, for $(\tilde \xi, \lambda)$
bounded from the origin $\Re \lambda \ge 0$, each column
of $\cS^{-X\to 0}T^-\EE^-$ is uniformly exponentially growing in $X$ at
rate at least
$$
e^{\mu_*X},
$$
where
$\mu_*(\tilde \xi, \lambda)$ is the smallest real part of the 
(positive real part) eigenvalues
of $\GG^-$ associated with the unstable subspace $\EE^-$.

By a similar argument, each column of $\cS^{-X\to 0}T^-\Pi_{\FF^-}\EE^0$
is uniformly exponentially decaying in $X$, at rate 
$e^{\mu^* X}$, where $\mu^*$ is the largest real part of
the (negative real part) eigenvalues of $\GG^-$ associated with the
stable subspace $\FF^-$.
Collecting information, we thus have
\begin{equation}\label{final}
\begin{aligned}
D_X(\tilde \xi, \lambda)&=
\beta(\tilde \xi, \lambda)
\Big(
D_X (\tilde \xi, \lambda) \cD(\tilde \xi, \lambda)
+O(e^{-\theta X})
+O(e^{(\mu_*-\mu^*) X})
\Big)\\
& \to
\beta(\tilde \xi, \lambda, X)D_X (\tilde \xi, \lambda) \cD(\tilde \xi, \lambda)
\end{aligned}
\end{equation}
as $X\to +\infty$, exponentially in $X$, where
\begin{equation}\label{beta}
\beta(\tilde \xi, \lambda, X):=
e^{\int_0^{-X} \Trace \GG(\tilde \xi, \lambda, z)dz}
e^{\Trace(\Pi_{\EE^-}\GG^-)X},
\end{equation}
for $(\tilde \xi, \lambda)$ uniformly bounded away from the origin
on $\Re \lambda \ge 0$.
\end{proof}

\begin{rem}\label{important}
\textup{
In general, the real parts of the eigenvalues of $\EE^-$ and $\FF^-$
can converge to zero as $(\tilde \xi, \lambda)$ approach the origin,
so that $\mu_*$, $\mu^*\to 0$ and the above convergence argument fails.
In the special case of a Lax $1$-shock, $\mu^*\to 0$, but
$\mu_*$ remains strictly negative \cite{Z3}, and so we obtain uniform
convergence by this argument
 on all of $\Re \lambda \ge 0$.  Indeed, in the one-dimensional
case, we obtain uniform convergence on $\Re \lambda \ge -\eta$, $\eta>0$.
(The only obstruction in the multi-dimensional case is that $\cD$ is 
not defined on this set \cite{Z3}.)
}
\end{rem}

Combining Remarks \ref{hfrems}, \ref{analytic}, and \ref{important},
we obtain the following simple result reducing determination
of one-dimensional stability
to computation of the stability index.

\begin{lem}\label{onezero}
Under (H0)--(H7) and 
stability of the constant layer $U\equiv U_-$,
for a stable Lax $1$-shock in one dimension, 
$D_X$ has exactly one zero on $\Re \lambda > -\eta$,
$\eta>0$, for $X$ sufficiently large, hence is stable
if and only if its stability index is positive.
\end{lem}

\begin{proof}
By Remarks \ref{hfrems}, under stability of $U\equiv U_-$,
we may restrict attention to a compact set of frequencies,
while by Remark \ref{analytic} we may take $D_X$ and $\cD$
analytic on $\Re \lambda \ge -\eta$.
As the uniform limit of analytic functions, we find that the number
of zeros of $\cD D_-$ on $\Re \lambda > -\eta$ is equal to
the number of zeros of $D_X$ for $X$ sufficiently large.
As the number of zeros of $\cD$ is one and the number of zeros
of $D_-$ is zero, by stability, the number of zeros of $D_X$
is one as asserted.

As the stability index $\Gamma$ counts the parity of the number of nonstable
roots $\Re \lambda>0$ (see \cite{GZ,Z3}), positivity of $\Gamma$
corresponding to an even number of nonstable roots, and
since $\Gamma=0$ corresponds to instability \cite{NZ1},
we thus obtain stability if and only if $\Gamma>0$.
\end{proof}

\section{Behavior near the origin}\label{origin}

\begin{proof}[Proof of Theorem \ref{conv} for $(\tilde \xi, \lambda)\to (0,0)$]
The shock evans function $\cD$ is continuous at the origin, with $\cD(0,0)=0$,
and $D_-$ is bounded on compact sets.
Thus, to complete the proof of Theorem \ref{conv}, it suffices to show that
$$
D_X(\tilde \xi,\lambda)/\beta(\tilde \xi,\lambda)\to 0
$$
as $(\tilde \xi, \lambda) \to 0$ and $X\to +\infty$, 
with $\beta$ defined as in \eqref{beta}.

Noting that $T^-\EE^-(0,0)$ contains by continuity
all exponentially decaying solutions of the one-dimensional
eigenvalue equation, hence, in particular, $\hat U'(x_1)$,
we may without loss of generality assign the value $\hat U'(x_1)$
to the first column of $T^-\EE^-$ at $(0,0)$.
Moreover, noting that the strongly unstable subspace of $\GG^-$,
defined as the part whose eigenvalues have strictly positive real
part even at $(0,0)$, perturbs analytically, we 
may restrict the first column to this subspace, ensuring that
the first column of $\cS^{-X\to 0}T^-\EE^-$ is analytic up to
the origin and moreover grows exponentially
in $X$, at rate $e^{\theta X}$, some $\theta>0$, for $|(\tilde \xi, \lambda)|$
sufficiently small, 
hence contributions to $D_X$ coming from the first
columns of the second two terms in the last line of \eqref{keycalc}
are exponentially small as $X\to 0$ and can be ignored, while the
contributions in other columns are at least bounded.
Likewise, we may arrange that the first column of $T^+\EE^+$
be analytic up to the origin and equal to $\hat U'$ at $(\tilde \xi, 
\lambda)=(0,0)$.

In place of \eqref{final}, therefore, we obtain the weaker estimate
\begin{equation}\label{weakfinal}
\begin{aligned}
|D_X(\tilde \xi, \lambda)|/\beta(\tilde \xi, \lambda)&\le
\Big|D_X (\tilde \xi, \lambda) \det\Big(T^-(\EE^-_1,O(1)), 
T^+\EE^+\Big)_{x_1=0}\Big|\\
&\quad
+O(e^{-\theta X})
+O(e^{(\mu_*-\mu^*) X})
\\
& \to 0
\end{aligned}
\end{equation}
as $(\tilde \xi, \lambda)\to (0,0)$,  $X\to +\infty$,
where $T^-\EE^-_1$ 
and $T^+\EE^+_1$ 
denote the first columns of 
$T^-\EE^-$ and $T^+\EE^+$, 
since by our choice of normalization
$T^-\EE^-_1$ and $T^+\EE^+_1$ are continuous (indeed, analytic)
at the origin and coincide for $(\tilde \xi, \lambda)=(0,0)$.
This completes the proof of the theorem.
\end{proof}

\begin{proof}[Proof of Theorem \ref{lim}]
Fix $\tilde \xi\equiv 0$, so that $D(0,\lambda)$ is continuous in $\lambda$.
Without loss of generality, take 
$D_-(0,0)=1$, with, moreover,
\begin{equation}\label{balance}
\EE^0(0,0) e^{-\tilde \EE^{-*}\GG^-  \EE^-X}
=T(-X)e^{-\GG^-X}\Big(\EE^-(0,0)+ \FF^-(0,0) \alpha\Big)
\end{equation}
for some $k\times k$ matrix $\alpha$, where $k=\dim \FF^-$.
Thus,
\begin{equation}\label{prekeycalc}
\begin{aligned}
\cS^{-X\to 0}\EE^0&=
T^-(0) \Big( e^{\GG^-X} T^-(-X)^{-1}\EE^0 \Big)\\
&= T^-(0)\Big(\EE^-(0,0)+ \FF^-(0,0) \alpha\Big)
e^{-\tilde \EE^{-*}\GG^-  \EE^-X}.
\end{aligned}
\end{equation}

Following the proof of Theorem \ref{conv}, we find that the
quantity
$$
D_X(0,0)/ e^{\int_0^{-X} \Trace \GG(0,0, z)dz}e^{\Trace \Pi_{\EE^-}\GG^- X}
$$
is given (exactly, with no exponentially decaying error) by
$$
\det \Big( T^-(\EE^-
+  \FF^- \alpha , T^+\EE^+\Big)|_{x_1=0,
\, (\tilde \xi,\lambda)=(0,0)},
$$
where $\Trace \GG(0,0, z)$ and $\Trace \Pi_{EE^-}\GG^-(0,0) $ are  real, hence
$$
e^{\int_0^{-X} \Trace \GG(0,0, z)dz}
\;\hbox{\rm and }\;
e^{\Trace \Pi_{\EE^-} \GG^- X}
$$
are real and positive.

Appealing to Proposition \ref{Elim}, we may arrange without loss
of generality that the first $k_1$ columns of
$T^-(0,0,x_1)e^{\GG^-(0,0) x_1} \EE^-(0,0)$ 
consist of functions $(b_{11}\phi_j,0)$, where 
$\phi_j$, $j=1, \dots, k_1$  are solutions of
\begin{equation}\label{1eval}
B_{11}\phi_j'-A_1\phi_j=0
\end{equation}
that are uniformly exponentially decaying as $x_1\to -\infty$, hence
exponentially growing in forward direction,
and the remaining $k_2$ columns consist of functions
$(*, r_j^-)$, $j=1,\dots,k_2$,
 where $r_j^-$ are constant eigenvectors of
$A_1^-$ with negative eigenvalues $a_j$,
and similarly for $T^+\EE^+$.
Likewise, we may arrange that the columns of $T^-\FF^-$
consist of $l_1$ solutions of \eqref{1eval}
that are uniformly exponentially decaying in $x_1$ in forward direction
and $l_2$ functions
$(*, r_j^-)$, $j=1,\dots,l_2$,
 where $r_j^-$ are constant eigenvectors of
$A_1^-$ with positive eigenvalues $a_j$.

Finally, we may choose the first column of $T^-\EE^-|_{x_1=0}$ as $\hat U'(0)$,
noting that $\hat U'(0)$ lies also in $T^+\EE^+|_{x_1=0}$.
Combining these facts, we find that, up to an exponentially decaying
error with respect to $X$, we may rewrite
$D_X(0,0)/ e^{\int_0^{-X} \Trace \GG(0,0, z)dz}e^{\Trace(\Pi_{\EE^-}\GG^-)X}$
as
\begin{equation}\label{blockdet}
\det \begin{pmatrix} 
* & \psi_2 & \cdots & \psi_{k_2} & * & *&
\psi_{k_1+1} & \cdots & \psi_{n+1} \\
|\hat U'(-X)|\hat V_X & 0 & \cdots & 0 & \hat \cR^-& \cR^+ & 0 &\cdots & 0
\end{pmatrix}
\end{equation}
evaluated at $x_1=0$, $(\tilde \xi,\lambda)=(0,0)$, where
$\psi_j:=b_{11}\phi_j$,
$\begin{pmatrix}* \\|\hat U'(-X)|\hat V_X \end{pmatrix}$ 
is the part of the first column of  $\FF^-\alpha$ 
involving only slow modes 
$$
((-a_j^-)^{-1}b_{11}r_j^-,r_j^-)
$$
(recall \eqref{Xconst}), and
$\begin{pmatrix}*\\\hat \cR^- \end{pmatrix}$ 
is the part involving only slow modes of the block 
of $\EE^-+\FF^-\alpha$ corresponding to the slow block
$\begin{pmatrix}*\\ \cR^- \end{pmatrix}$ of $\EE^-$.
Referring to \eqref{balance}, we see that $\hat V_X$
has a limit $\hat V$ as $X\to +\infty$ so long as
$S:=\lim_{z\to -\infty}(\hat U'/|\hat U'|)(z)$ exists
(the first column of $T^-(-X)e^{-\GG^-X}\EE^-$ being
then approximately $|\hat U'(-X)|S$ as $X\to +\infty$,
so that the slow component of the first column of $\FF^-\alpha$
is approximately $|\hat U'(-X)|\hat V$ for a fixed $\hat V$
determined by $S$).

By a block determinant expansion, we have, therefore,
$$
\begin{aligned}
D_X&(0,0)/ e^{\int_0^{-X} \Trace \GG(0,0, z)dz}e^{\Trace(\Pi_{\EE^-}\GG^-)X}=\\
&\quad
\sigma |\hat U'(-X)| \det ( \psi_2, \dots, \psi_{n+1}) 
\det(\hat \cR^-, \cR^+, \hat V_X )|_{x_1=0, \, (\tilde \xi,\lambda)=(0,0)}\\
&\to 
\sigma |\hat U'(-X)| \det ( \psi_2, \dots, \psi_{n+1}) 
\det(\hat \cR^-, \cR^+, \hat V )|_{x_1=0, \, (\tilde \xi,\lambda)=(0,0)},\\
\end{aligned}
$$
where $\sigma=\pm 1$ depending on dimensions $n$, $r$ and
$\det ( \psi_2, \dots, \psi_{n+1}) \ne 0$ assuming stability
of the limiting shock $\hat U$ (else $\cD$ would vanish to second 
instead of first order at the origin \cite{Z3}).
Normalizing 
$$
 \sgn \sigma\det ( \psi_2, \dots, \psi_{n+1}) =+1,
$$
we obtain
$$
\lim_{X\to +\infty} \sgn \lim_{\lambda\to 0^+ \, 
\hbox{\rm real}} D_X(0,\lambda)= \hat \delta:=
\det (\hat R^-,R_+,\hat V)
$$
as claimed.

Reviewing the computation in \cite{Z3} of $\partial_\lambda\cD(0,0)$,
we find that this is the same normalization of $\psi_j$ columns leading
to the assumed normalization \eqref{shocknorm}, whence
\eqref{1stabcond} is necessary for stability by the discussion above
the statement of the theorem.

Finally, for a Lax $1$-shock, all eigenvalues of $A_1^-$ are positive,
hence the $\cR^-$ block is empty in the computation above and the boundary
conditions must be of inflow type.
If also, the boundary conditions are Dirichlet type, then by
\eqref{dirkerin}, the first column of 
$\EE^0(0,0) e^{-\tilde \EE^{-*}\GG^-  \EE^-X}$
(since all of $\EE^0$) 
is of form
$$
\Big(\EE^0(0,0) e^{-\tilde \EE^{-*}\GG^-  \EE^-X}\Big)_1=
\hat \EE^-_1 +\Big(\hat \FF^{-,slow}\alpha\Big)_1=
\begin{pmatrix} 0 \\ \begin{pmatrix}0 \\ *\\ \end{pmatrix} \end{pmatrix},
$$
where $\hat \EE^-_1$ and $\hat \FF^{-,slow}_1$, defined as the first column
of $\hat \EE^-:=T^-(-X)e^{-\GG^-X}\EE^-$ and  the slow part of 
$\hat \FF^-:=T^-(-X)e^{-\GG^-X}\FF^-$,
are asymptotically of form
(since $T^-e^{\GG^-x_1}\EE^-_1= (b_{11}\hat U'(x_1),0)$ and
$\lim_{z\to -\infty} T^-=Id$)
$$
|\hat U'(-X)|\begin{pmatrix} b_{11}^-S \\ 0\\ \end{pmatrix}
\; \hbox{\rm and}\;
|\hat U'(-X)|\begin{pmatrix}  b_{11}^-(-A_1^-)^{-1}\hat \cR_1^-
\\ \hat \cR_1^- \end{pmatrix}
$$
as $X\to +\infty$,
where $\hat \cR_1^-$ lies in the unstable subspace of $A_1^-$.
Equating, we find that $\hat \cR_1^-=\begin{pmatrix}
0 \\ c\end{pmatrix}$
and
$b_{11}^-S= b_{11}^-(A_1^-)^{-1}\hat \cR_1^-$.
Recalling that $\Big(A_1^{11}, A_1^{12}\Big) \hat U'=0$
by the linearized boundary-layer ODE, so that 
$$
\Big((A_1^-)^{11}, (A_1^-)^{12}\Big) S=0,
$$
we find by inspection that 
$\hat \cR_1= A_1^- S$,
or, in the notation of the general case,
\begin{equation}\label{V}
\hat V= A_1^-S,
\end{equation}
yielding the result.
\end{proof}

\begin{proof}[Proof of Theorem \ref{1lim}]
Immediate, by Lemma \ref{onezero} and Theorem \ref{lim}.
\end{proof}

\begin{proof}[Proof of Theorem \ref{mnec}]
Set $\eta:=|\hat U'(-X)|$, and consider 
$0\le \rho \le Ce^{-\theta X}$ for some
fixed $C>0$ and $\theta>0$ sufficiently small, in particular,
small enough that 
\begin{equation}\label{smallness}
|\hat U'(-X)|<<Ce^{-\theta X},
\end{equation}
setting
$$
(\tilde \xi, \lambda, \eta)=:
(\rho \tilde \xi_0, \rho \tilde \lambda_0, \rho \eta_0)
$$
with
$|(\tilde \xi_0, \tilde \lambda_0, \eta_0)|=1$.
Note, as $X\to +\infty$, that the set of possible values of
$(\tilde \xi_0, \tilde \lambda_0, \eta_0)$ expands to
the full positive half-sphere $\Re \lambda_0$, $\eta_0>0$.
Restrict now to a compact subset of the positive half-sphere,
recalling (see Proposition \ref{cbases}) that the Evans functions
$D_X$ and $\cD$, and their component columns, are $C^\infty$
in $\rho$, $\tilde \xi_0$, $\tilde \lambda_0$.

Within the specified parameter-regime, both slow and fast modes
of \eqref{firstorder}
at $(\tilde \xi, \lambda)=\rho(\tilde \xi_0, \lambda_0)$
are well-approximated on $x_1\in [-X, 0]$ by their limiting
values as $\rho\to 0$, described in Proposition \ref{Elim}.
Mimicking the one-dimensional computations \eqref{blockdet},
\eqref{V}, we may rewrite
$$
D_X(\rho \tilde \xi_0,\rho \lambda_0)/ 
e^{\int_0^{-X} \Trace \GG(\rho \tilde \xi_0,\rho \lambda_0, z)dz}
e^{\Trace(\Pi_{\EE^-}\GG^-)X}
$$
as the sum of
$$
\begin{aligned}
\rho \det & \begin{pmatrix} 
* & \psi_2 & \cdots & \psi_{k_2} & * & *&
\psi_{k_1+1} & \cdots & \psi_{n+1} \\
\eta_0 A_1 S & 0 & \cdots & 0 & \hat \cR^-& \cR^+ & 0 &\cdots & 0
\end{pmatrix}\\
&\quad 
+o(\rho \eta_0)
\end{aligned}
$$
and
$$
\begin{aligned}
&\rho \det  \begin{pmatrix} 
* & \psi_2 & \cdots & \psi_{k_2} & * & *&
\psi_{k_1+1} & \cdots & \psi_{n+1} \\
\partial_\rho T^-\EE^-_1 -\partial_\rho T^+\EE^+_1 
& 0 & \cdots & 0 & \hat \cR^-& \cR^+ & 0 &\cdots & 0
\end{pmatrix}\\
&\qquad\quad 
+o(\rho |(\tilde \xi_0,\lambda_0)|),\\
\end{aligned}
$$
both evaluated at $x_1=0$, $\rho=0$, and $(\tilde \xi_0,\lambda_0)$. 
We omit the details of this straightforward but tedious computation.

A standard computation \cite{ZS,Z3,GMWZ4} using the variational
equations of \eqref{firstorder} with respect to $\rho$ yields
$$
\partial_\rho T^-\EE^-_1 -\partial_\rho T^+\EE^+_1 =
 \lambda_0[U]+\sum_{j=2}^{d}i\tilde \xi_0^j[F_j(U)],
$$
whence, normalizing as usual so that
$$
 \sgn \sigma\det ( \psi_2, \dots, \psi_{n+1}) =+1,
$$
we obtain by block determinant expansion
\begin{equation}\label{expansion}
\begin{aligned}
\rho^{-1}D_X(\rho \tilde \xi_0,\rho \lambda_0)/ 
e^{\int_0^{-X} \Trace \GG(\rho \tilde \xi_0,\rho \lambda_0, z)dz}
&e^{\Trace(\Pi_{\EE^-}\GG^-)X}=\\
&\hat \Delta(\tilde \xi_0, \lambda_0, \eta_0)+o(1),
\end{aligned}
\end{equation}
where $\hat \Delta$ is defined as in \eqref{hatDelta}
and $o(1)$ is $C^1$ with respect to
$\tilde \xi_0$, $\lambda_0$, and $\eta_0$ for each fixed $X$
and $\to 0$ uniformly as $X\to 0$.

By an application of the Implicit Function Theorem, it follows
that existence of a root $(\tilde \xi_0^*, \lambda_0^*,\eta_0^*)$
of $\hat \Delta$ on $\Re \lambda_0>0$ at which 
$\partial_{\lambda_0}\hat \Delta\ne 0$ implies existence of a nearby
root $(\tilde \xi_0^*, \lambda_0^\dagger, \tilde \eta_0^*)$,
$\Re \lambda_0^\dagger>0$, $\rho>0$,
of
$$
\rho^{-1}D_X(\rho \tilde \xi_0,\rho \lambda_0)/ 
e^{\int_0^{-X} \Trace \GG(\rho \tilde \xi_0,\rho \lambda_0, z)dz}
e^{\Trace(\Pi_{\EE^-}\GG^-)X},
$$
hence of $D_X$, for $X$ sufficiently large, or instability
of $\bar U^X$.
Thus, nonvanishing of $\hat \Delta$ on the strictly positive
half-sphere is necessary for stability as $X\to 0$.
\end{proof}

\begin{proof}[Proof of Theorem \ref{msuff}]
The estimate \eqref{expansion} in fact holds 
for all $\eta:=|\hat U'(-X)|$ and $0\le \rho \le Ce^{-\theta X}$,
with the $o(1)$ term uniformly decaying and uniformly $C^0$
as $X\to \infty$ (however, not uniformly $C^1$; see \cite{GMWZ4,GMWZ5,GMWZ6}).
It follows therefore, that nonvanishing of $\hat \Delta$ on 
the (closed) nonnegative half-sphere, implying a lower bound on
$|\hat \Delta|$, implies nonvanishing of $D_X$ on the parameter range
$0\le \rho \le Ce^{-\theta X}$, for $X$ sufficiently large.

If $Ce^{-\theta X}\le \rho << 1$, on the other hand,
a much cruder estimate yields
$$
\begin{aligned}
\rho^{-1}D_X(\rho \tilde \xi_0,\rho \lambda_0)&/ 
e^{\int_0^{-X} \Trace \GG(\rho \tilde \xi_0,\rho \lambda_0, z)dz}
e^{\Trace(\Pi_{\EE^-}\GG^-)X}
=\\
&\qquad
\Delta(\tilde \xi_0, \tilde \lambda_0) 
+ o(1) + O(|\hat U'(-X)|/Ce^{-\theta X})=\\
 &\qquad
\Delta(\tilde \xi_0, \tilde \lambda_0) + o(1),
\end{aligned}
$$
by \eqref{smallness}, again with $o(1)$ uniformly decaying
as $X\to +\infty$.
This implies nonvanishing of $D_X$ on the parameter range
$1>>\rho \ge Ce^{-\theta X}$, for $X$ sufficiently large.

For $\rho$ bounded from below, on the other hand, we have
by the basic convergence result of Theorem \ref{conv} that
$D_X$ is nonvanishing if $\cD$ and $D_-$ are nonvanishing,
i.e., if $\hat U$ and $U\equiv U_-$ are stable.
This completes the proof of the theorem.
\end{proof}

\section{Application to gas dynamics}\label{sec:apps}

We now apply our results to the fundamental example of
compressible gas dynamics, restricting
without loss of generality (by rotational invariance of the
equations) to dimension $d=2$.
Consider the compressible Navier--Stokes equations 
\begin{subequations}\label{eq:ns}
\begin{equation}
\rho_t+ (\rho u)_x +(\rho v)_y=0,\label{eq:mass}
\end{equation}
\begin{equation}
(\rho u)_t+ (\rho u^2+p)_x + (\rho uv)_y=
(2\mu +\eta) u_{xx}+ \mu u_{yy} +(\mu+ \eta )v_{xy},
\label{eq:momentumx}
\end{equation}
\begin{equation}
(\rho v)_t+ (\rho uv)_x + (\rho v^2+p)_y=
\mu v_{xx}+ (2\mu +\eta) v_{yy} + (\mu+\eta) u_{yx},
\label{eq:momentumy}
\end{equation}
\begin{multline}\label{eq:energy}
(\rho E)_t+ (\rho uE+up)_x + (\rho vE+vp)_y\\
=\Big( \kappa T_x + (2\mu+\eta)uu_x + \mu v(v_x+u_y) + \eta uv_y\Big)_x \\
+\Big( \kappa T_y+ (2\mu+\eta)vv_y + \mu u(v_x+u_y) + \eta vu_x\Big)_y
\end{multline}
\end{subequations}
on the half-plane $x\in \RR^+$, $y\in \RR$, 
where $\rho$ is density, $u$ and $v$ are velocities in $x$ and $y$
directions, $p$ is pressure, 
\begin{equation}
E=e+\frac{u^2}{2} +\frac{v^2}{2} .\label{eq:internal_kinetic}
\end{equation}
is total energy density,
$e$ and $T$ are internal energy density and temperature,
and constants $\mu>|\eta|\ge0$ and $\kappa>0$
are coefficients of first (``dynamic'') 
and second viscosity and heat conductivity.

We assume ideal  (``$\gamma$-law'') gas equations of state
\begin{equation}
p_0(\rho,T)= \Gamma \rho e,\quad e_0(\rho,T)=c_vT,
\label{eq:ideal_gas}
\end{equation}
where $c_v>0$ is the specific heat at constant volume,
$\Gamma := \gamma -1> 0$, and
$\gamma 
> 1$ is the adiabatic index
of the gas; 
equivalently,
\begin{equation}\label{rhoS}
p(v,S)=ae^{S/c_v}\rho^{\gamma}, 
\end{equation}
where $S$ is thermodynamical entropy \cite{Ba,Sm}.
In the notation of \eqref{sys}, we have 
\begin{equation}\label{consvar}
\tilde U=(\rho, \rho u, \rho v, \rho E)=:(\rho, m_1,m_2, \cE)
\end{equation}
and
\begin{equation}\label{Fvar}
F_1(\tilde U)=
(\rho u, \rho u^2 +p, \rho v^2, \rho u E)=
(m_1, m_1^2/\rho +p, m_1m_2/\rho, m_1\cE/\rho).
\end{equation}

\begin{rem}\label{values}
\textup{
In the thermodynamical rarified gas approximation, 
\begin{equation}\label{gammaformula}
\gamma=\frac{2n+3}{2n + 1},
\qquad
\nu/\mu= \frac{9\gamma -5}{4}
\qquad
\eta=-\frac{2}{3}\mu
\end{equation}
for $\nu:=\kappa/c_v$,
where $n$ is the number of constituent atoms of gas molecules (here
assumed to have ``tree'' structure) \cite{Ba}, with
$\gamma= 5/3$ and $\gamma= 7/5$ for the main applications of
monatomic and diatomic gas.  In particular, 
\begin{equation}\label{rat}
1<\gamma <2 \quad \hbox{\rm and}\quad \frac{\nu}{2\mu+\eta} >1
\end{equation}
for common gases, a conclusion that is born out by 
experiment.  See Appendices A and B of \cite{HLyZ1} for further discussion.
}
\end{rem}

\subsection{Viscous Shock Profiles}\label{sec:viscous}

From \eqref{eq:ns}, setting time-derivatives to zero, integrating in
$x$, and rearranging, we obtain after a brief calculation
the standing-shock ODE
\begin{equation}\label{idealode}
\begin{aligned}
\hat u'= (2\mu+\eta)^{-1}\Big(m (\hat u-u_-) + \Gamma(\hat \rho \hat e- 
\hat \rho_-e_-)\Big),\\
\hat e'= \nu^{-1}\Big( m  \big(\hat e- e_-\big)
- \frac{m(\hat u-u_-)^2}{2} +(\hat u-u_-)\Gamma\hat \rho_- T_-\Big),
\end{aligned}
\end{equation}
where $m:=\hat \rho \hat u\equiv \const$ and $\hat v\equiv \const$.

Using various scale-invariances of system \eqref{eq:ns},
we may take without loss of generality $m=\rho_-=u_-=1$,
$v_-=v_+=0$, yielding
\begin{equation}\label{midealode}
\begin{aligned}
\hat u'= \frac{1}{2\mu+\eta}\left( (\hat u-1) + 
\Gamma \left(\frac{\hat e}{\hat u}- e_-\right)\right),\\
\hat e'= \nu^{-1}\left((\hat e- e_-)
- \frac{(\hat u-1)^2}{2} +(\hat u-1)\Gamma e_-\right)
\end{aligned}
\end{equation}
with $\hat v\equiv 0$, with endstates
\begin{equation}\label{endstates}
e_+=\frac{u_+\alpha(u_+-1)}{\Gamma(\Gamma+2-\alpha)}, 
\quad
e_-=\frac{(u_+-1)(\Gamma+2)}{\Gamma(\Gamma+2-\alpha)},\quad
\rho_+=1/u_+,
\end{equation}
$\alpha:=\frac{\Gamma+2-\Gamma u_+}{u_+-u_*}$,
parametrized by the single quantity
\begin{equation}\label{phys}
1\ge u_+> u_*:= \frac{\Gamma}{\Gamma + 2}.
\end{equation}
In the strong-shock limit $u_+\to u_*$, $e_-\to 0$, with all other
quantities remaining in physical range;
for details of these computations, see \cite{HLyZ2}, Sections 3--5.

Linearizing \eqref{midealode} about $(u_-,e_-)=(1,e_-)$, we obtain
\begin{equation}\label{linode}
\begin{pmatrix} u\\e \end{pmatrix}'=M_- \begin{pmatrix} u\\e \end{pmatrix},
\qquad
M_-:=
\begin{pmatrix} \frac{1}{2\mu+\eta} & 0\\
 0& \frac{1}{\nu} \end{pmatrix}
\begin{pmatrix} 
1-\Gamma e_- & \Gamma \\ 
\Gamma e_- & 1 \end{pmatrix},
\end{equation}
determining the asymptotic behavior of $(\hat u,\hat e)(z)$ as
$z\to -\infty$.
One may check for all $1\ge u_+> u_*$ that $M_-$ has two positive
distinct real eigenvalues $0< \omega_-\le 1/\nu\le \omega_+$,
$$
\omega_\pm=
\frac{1}{\nu}+
\frac{ \Big( \frac{1-\Gamma e_-}{2\mu+\eta} - \frac{1}{\nu} \Big)
\pm \sqrt{ 
\big( \frac{1-\Gamma e_-}{2\mu+\eta} -\frac{1}{\nu} \big)^2
+ \frac{4\Gamma e_-}{(2\mu+\eta)\nu} }}{2},
$$
with associated eigenvectors 
$s_j=(-1, -\frac{\Gamma e_-}{\nu(\omega_j-1/\nu)})^T$,
merging in the special limiting case $u_+\to u_*$/$e_-\to 0$, $2\mu+\eta=\nu$
to a pair of real semisimple eigenvalues.

That is, for a Lax $1$-shock,
 $U_-$ is a repellor for the standing-wave ODE, and $U_+$ a saddle, 
in agreement with the abstract conclusions of \cite{MaZ3} for extreme shocks
of general systems and of \cite{Gi} for shock
profiles of gas dynamics with general equation of state.
In particular, note that 
$$
\det M_-=(\nu (2\mu+\eta))^{-1}
(1- \Gamma(1+\Gamma) e_-)>0,
$$
with $(1- \Gamma(1+\Gamma) e_-)$ approaching $1$ in the strong shock
limit $u\to u_*$/$e_-\to 0$, and $1- \frac{2(\Gamma+1)}{(\Gamma +2)^2}>
\frac{\Gamma^2+2}{(\Gamma +2)^2}$
in the weak shock limit $u_+\to 1$.\footnote{
This repairs an error of \cite{SZ}, in which $U_-$ was mistakenly 
computed to be a saddle, leading to an incorrect value of $S$.}
By reality and simplicity of the eigenvalues $\omega_j$, we have that limits
\begin{equation}\label{sform}
s:=\lim_{z\to -\infty} (\hat u',\hat e')/ |(\hat u', \hat e')|)
\end{equation}
and
\begin{equation}\label{Sform}
\begin{aligned}
S:=\lim_{z\to -\infty} (\hat U'/|\hat U'|)
&=\frac{\partial  U}{\partial ( u,  e)}|_{U_-} s
=
\begin{pmatrix}
-1&0\\
0&0\\
0&0\\
1/2- e_-& 1
\end{pmatrix}
s\\
\end{aligned}
\end{equation}
exist, with $s$ generically lying parallel to the slow mode $s_-$,
or
\begin{equation}\label{genS}
\begin{aligned}
S&= \begin{pmatrix}
-1&0\\
0&0\\
0&0\\
1/2- e_-& 1
\end{pmatrix}
\begin{pmatrix} -1 \\ -\frac{\Gamma e_-}{\nu(\omega_--1/\nu)} \end{pmatrix}\\
&=
\begin{pmatrix} 1 \\ 0 \\ 0 \\
 e_- \frac{\Gamma -\nu(\omega_--1/\nu)}{\nu(\omega_--1/\nu)}
-\frac{1}{2} \end{pmatrix}.\\
\end{aligned}
\end{equation}

Finally, from \eqref{consvar}--\eqref{Fvar}, we obtain
after a brief calculation
\begin{equation}\label{A1-}
\begin{aligned}
A_1^-:= \partial (F_1/\partial \tilde U)(U_-)
&=
\begin{pmatrix}
0 & 1 & 0 & 0\\
p_\rho -1 +p_e/2 & 2-p_e & 0 & p_e\\
0 & 0 & 1 & 0\\
-1/2 & 1/2 & 0 & 1\\
\end{pmatrix}\\
&=
\begin{pmatrix}
0 & 1 & 0 & 0\\
\Gamma e_- -1 +\Gamma/2 & 2-\Gamma/2 & 0 & \Gamma\\
0 & 0 & 1 & 0\\
-1/2 & 1/2 & 0 & 1\\
\end{pmatrix},
\end{aligned}
\end{equation}
from which we compute
\begin{equation}\label{AS}
A_1^-S =
\begin{pmatrix} 0 \\
\Gamma e_- -1 +\Gamma/2 
+ \Gamma e_- \frac{\Gamma -\nu(\omega_--1/\nu)}{\nu(\omega_--\Gamma /\nu)}
\\ 0\\
 e_- \frac{\Gamma -\nu(\omega_--1/\nu)}{\nu(\omega_--1/\nu)} -1 \end{pmatrix}.
\end{equation}

\begin{equation}\label{[U]}
[U]= \Big(\frac{1-u_+}{u_+}, 0, 0, \frac{1-u_+}{2} \Big)^T.
\end{equation}

\subsubsection{The strong shock limit}
For $\frac{\nu}{2\mu+\eta}<1$, \eqref{genS} converges to
$S=(1,0,0,-1/2)^T$ in the strong shock limit $e_-\to 0$.
For $\frac{\nu}{2\mu+\eta}\ge 1$, however, 
\eqref{genS} becomes singular in the limit as $e_-\to 0$, 
for which also $\omega_-\to 1/\nu$.
To evaluate this limit, it is easier to return to \eqref{linode}
and compute directly with $e_-=0$, to obtain
$s\to (-1, 1-\phi )^T$, 
yielding the general formula
\begin{equation}\label{genlim}
\qquad
S\to (1,0,0, 1/2- \min\{1, \phi\} )^T
\; \hbox{as}\; u_+\to u_-,
\quad
\phi:=\frac{2\mu+\eta}{\nu}.
\end{equation}

Noting that 
\begin{equation}\label{limA}
A_1-\to
\begin{pmatrix}
0 & 1 & 0 & 0\\
 -1 +\Gamma/2 & 2-\Gamma/2 & 0 & \Gamma\\
0 & 0 & 1 & 0\\
-1/2 & 1/2 & 0 & 1\\
\end{pmatrix},
\end{equation}
we thus have
\begin{equation}\label{limAS}
A_1S-\to
\begin{pmatrix}
0\\ 
 \Gamma \max\{0,1-\phi\}-1 \\
0 \\
- \min\{1, \phi \} 
\\
\end{pmatrix}
\end{equation}
and
\begin{equation}\label{lim[U]}
[U]\to \Big(\frac{2}{\Gamma}, 0, 0, \frac{1}{\Gamma+2} \Big)^T,
\end{equation}
completing our asymptotic analysis.

\subsection{One-dimensional stability}\label{sec:oneapp}

Following the treatment in \cite{Se,SZ}, we note
for Lax $1$-shocks that 
\begin{equation}\label{alt}
\det (\cR_+, f)=\ell_+ \cdot f,
\end{equation}
for any vector $f\in \CC^n$, 
where $\ell_+$ is the unique stable left eigenvector of 
$\cA_+$ and $\cdot$ denotes complex
inner product.
In one dimension, $\ell_+$ is just the stable left eigenvector
of $A_+$, which may be computed to be
\begin{equation}\label{1ell}
\begin{aligned}
\ell_+&= 
\Big(p_\rho + cu + \frac{p_e(u^2/2-e)}{\rho},  -\frac{p_eu}{\rho}-c, 
0, \frac{p_e}{\rho}\Big)^T (U_+)\\
&=
\Big(c_+u_+ + \Gamma u_+^2/2,  -\Gamma u_+-c_+, 0, \Gamma \Big)^T,\\
\end{aligned}
\end{equation}
where
\begin{equation}\label{c}
c:=
\sqrt{pp_e/\rho^2 +p_\rho}
=
\sqrt{\Gamma (\Gamma+1)e}
\end{equation}
denotes sound speed.
This computation is most easily accomplished
by working in the more convenient nonconservative
coordinates $(\rho, u,v,e)$,
which are related to conservative variables $(\rho, \rho u, \rho v, \rho
(e+u^2/2+v^2/2))$ by a readily computed lower triangular change of coordinates;
see \cite{Se} or Appendix \ref{ellcomp}.

Combining all facts, we have
\begin{equation}\label{hatdeltaform}
\begin{aligned}
\hat \delta&= \ell_+\cdot A_-S \\
&=\Big(-\Gamma u_+-c_+\Big)
\Big(\Gamma e_- -1 +\Gamma/2 + 
\Gamma e_- \frac{\Gamma -\nu(\omega_--1/\nu)}{\nu(\omega_--\Gamma /\nu)}\Big)
\\
&\quad
+
\Gamma
\Big(e_- \frac{\Gamma -
\nu(\omega_--1/\nu)}{\nu(\omega_--1/\nu)} -1 \Big).
\end{aligned}
\end{equation}
It is readily verified on the other hand that 
\begin{equation}\label{sgnnorm}
\delta>0;
\end{equation}
see Section \ref{1dstablim} just below.
The one-dimensional stability condition \eqref{1Dstabcond}
thus reduces in this case to 
\begin{equation}\label{ns1stab}
\sgn \hat \delta>0,
\end{equation}
a condition that can be readily checked numerically using \eqref{hatdeltaform}.


\subsubsection{The strong shock limit}\label{1dstablim}
In the strong shock limit $u_+\to u_*$, 
we have $e_-\to 0$, 
$\alpha\to +\infty$, and 
$$
e_+\to u_*(1-u_*)/\Gamma=2/(\Gamma+2)^2 , 
$$
so that $c_+\to \sqrt{2\Gamma(\Gamma+1)}/(\Gamma+2)$ and
$$
\ell_+\to \Big(
\frac{\Gamma \sqrt{2\Gamma(\Gamma+1)}}{(\Gamma+2)^2}
+ \frac{\Gamma^3 }{2(\Gamma+2)^2},
-\frac{\Gamma^2 }{\Gamma+2}- \frac{\Gamma \sqrt{2\Gamma(\Gamma+1)}}{\Gamma+2}, 
0, \Gamma \Big)^T,
$$
hence, by \eqref{limAS},
\begin{equation}\label{limdel}
\begin{aligned}
\hat \delta&= \ell_+\cdot A_-S\\
&\to 
\Big(
-\frac{\Gamma^2 }{\Gamma+2}- 
\frac{\Gamma \sqrt{2\Gamma(\Gamma+1)}}{\Gamma+2}\Big) 
 \Big(\Gamma \max\{0,1-\phi \}-1 \Big)
-\Gamma \min\{1, \phi \}.
\end{aligned}
\end{equation}
Meanwhile,
$$
\delta=\ell_+\cdot [U]\to
\frac{2\sqrt{2\Gamma(\Gamma+1)}}{(\Gamma+2)^2}
+ \frac{\Gamma }{(\Gamma+2)^2}
+ \frac{\Gamma}{\Gamma+2}>0,
$$
from which we may conclude by homotopy/nonvanishing of $\delta$ 
that $\delta>0$ for all $1\ge u_+\ge u_*$, verifying 
\eqref{sgnnorm}--\eqref{ns1stab}.

{\bf The case $\phi \ge 1$.}
For $\phi \ge 1$, \eqref{ns1stab} becomes
$$
\Big(
\frac{\Gamma \sqrt{2\Gamma(\Gamma+1)}}{\Gamma+2}
+ \frac{\Gamma^2 }{\Gamma+2}\Big)
-\Gamma  >0,
$$
or
$$
\sqrt{2\Gamma(\Gamma+1)}> 2,
$$
which evidently fails for $\Gamma $ in the kinetic range
$0\le \Gamma \le 1$ (indeed, for all $\Gamma $ outside $(1,2)$).
Thus, we may conclude instability in the strong shock limit in this range.

{\bf The case $\phi \le 1$.}
For $\phi \le 1$, \eqref{ns1stab} becomes
$$
\Big(
-\frac{\Gamma^2 }{\Gamma+2}- 
\frac{\Gamma \sqrt{2\Gamma(\Gamma+1)}}{\Gamma+2}\Big) 
 \Big(\Gamma (1-\phi )-1 \Big)
-\Gamma  \phi  >0,
$$
or
$$
\Big(
\sqrt{2\Gamma(\Gamma+1)} +\Gamma \Big) 
 \Big(1-\Gamma (1-\phi ) \Big)
-(\Gamma+2)  \phi  >0.
$$
Defining 
$\sigma:= \sqrt{2\Gamma(\Gamma+1)} + \Gamma $,
we may rewrite this as
$$
 \sigma(1-\Gamma  )
>(\Gamma- \Gamma \sigma +2)  \phi ,
$$
or, assuming 
$\Gamma(1-\sigma) +2>0$, as holds for example on the kinetic
range $0<\Gamma < 1$, or $1<\gamma<2$ (on which $\sigma < 2+\Gamma$,
so $2+\Gamma > \Gamma \sigma$), as
\begin{equation}\label{limcrit}
\phi  <
 \frac{\sigma(1-\Gamma  )}{
\Gamma(1-\sigma) +2},
\end{equation}
which is satisfied for $\phi $ small enough, but for $\phi =1$,
hence for $\phi \le 1$ large enough, is not satisfied, by the
analysis of case $\phi =1$ above.

{\bf Common gases and the kinetic approximation.}
Recall that for common gases, $\phi $ is less than one.
For gases obeying the kinetic
approximation \eqref{gammaformula}--\eqref{rat},
\begin{equation}\label{phiform}
\phi =\frac{16}{27 \Gamma + 12},
\end{equation}
so that $\phi <1$ for $\Gamma \ge 4/27\approx  .148$, in particular
for $n$-atomic gases with $n\le 5$.
Thus, it is the case $\phi \le 1$ that is relevant to
typical applications.
Substituting \eqref{phiform} into \eqref{limcrit}
and noting that $\sqrt{2\Gamma (\Gamma+1)}\le \Gamma +1$
for $0<\Gamma<1$ yields the necessary condition
\begin{equation}\label{kincrit}
16(\Gamma +2)< (2\Gamma +1)(1+15\Gamma),
\end{equation}
or $0<(\Gamma -1)(30\Gamma +31)$, which is violated
for the entire kinetic range $0<\Gamma < 1$.

{\bf Conclusions}
By Remarks \ref{SZrmk}.2 and \ref{mSZrmk}.2, boundary layers
are both one- and multi-dimensionally stable in the standing
shock limit for limiting shocks of sufficiently small amplitude,
i.e., $1-u_+$ sufficiently small.
By the calculations above, however, for typical gas laws,
they are not even one-dimensionally stable in the strong shock 
limit for limiting shocks of sufficiently large amplitude,
i.e., $u_+-u_*$ sufficiently small, even though the corresponding
shock is perfectly stable \cite{HLyZ1,HLyZ2}.

Thus, we have the striking conclusion that
{\it for (all!) 
typically physically occurring gases under inflow Dirichlet boundary
conditions, there is a transition from stability to instability
of boundary layers in the standing shock limit as the amplitude
of the limiting shock increases from zero to its maximum value.}

\subsection{Multi-dimensional stability}\label{sec:mapp}

The computation of $\ell_+(\tilde \xi, \lambda)$ in multi-dimensions
may be found, for example, in Appendix C, \cite{Z3}\footnote{
Contributed by K. Jenssen and G. Lyng}, where it is computed as
\begin{equation}\label{gmell}
\ell_+(\tilde \xi, \lambda)=
\Big(
\theta - \frac{ic\beta u}{\sqrt{\tilde \xi^2-\beta^2}} +\frac{\eta u^2}{\beta},
\frac{ic\beta }{\sqrt{\tilde \xi^2-\beta^2}} -\eta u,
\frac{c\tilde \xi}{\sqrt{\tilde \xi^2-\beta^2}},
\eta
\Big)_+,
\end{equation}
where
$\theta:=p_\rho -\frac{p_ee}{\rho}= 2\Gamma e$,
$\eta:=\frac{p_e}{\rho}= \Gamma$,
and $c$ is sound speed \eqref{c},
or
\begin{equation}\label{mell}
\ell_+(\tilde \xi, \lambda)=
\Big(
2\Gamma e_+ - \frac{ic_+\beta_+ u_+}{\sqrt{\tilde \xi^2-\beta_+^2}} +\frac{\Gamma u_+^2}{\beta_+},
\frac{ic_+\beta_+ }{\sqrt{\tilde \xi^2-\beta_+^2}} -\Gamma u_+,
\frac{c_+\tilde \xi}{\sqrt{\tilde \xi^2-\beta_+^2}},
\Gamma
\Big),
\end{equation}
where
\begin{equation}\label{betaform}
\beta:=\frac{-u\lambda - \sqrt{\lambda^2 
+\tilde \xi^2(c^2-u^2)}}{c^2-u^2}.
\end{equation}
Together with our computation of $A_1^-S$ in \eqref{AS},
this determines $\ell_+\cdot A_1^1S$.
Meanwhile, $\Delta:=\ell_+\cdot (\lambda[U]+i\tilde \xi[F_2(U)])$
is computed for the same choice of $\ell_+$ in
Appendix C, \cite{Z3} (equation displayed below C.36),
thus determining $\hat \eta(\tilde \xi, \lambda)=
\ell_+\cdot A_1^1S/\Delta(\tilde \xi, \lambda).$

With Proposition \ref{verprop}, this gives a straightforward
means of numerical determination of multidimensional instability,
by plotting the image of $\hat \eta(1, i\tau)$ as $\tau$ ranges
over the real axis and checking whether or not this curve strikes
the nonnegative real axis; however, we shall not carry this out here.

The numerical determination of one- and multi-dimensional stability
transitions for ideal and other gas laws would be interesting
problems for further investigation.
A further very interesting open open problem is to 
determine analytically the stability transitions
as was done for the inviscid shock problem (involving only
$\Delta$) in \cite{Er,M}; see Appendix C, \cite{Z3}.

%
%
%

\appendix
\section{Computation of $\ell_+$ in one dimension}\label{ellcomp}

In this appendix, we carry out for completeness the computation
of $\ell_+$ for the one-dimensional Navier--Stokes equations,
verifying \eqref{1ell}.
In variables $(\rho, u,v,e)$, the quasilinear hyperbolic part
of the equations becomes
\begin{align}
\begin{split}
& \rho_t+ q\cdot \nabla \rho + \rho \div q =0,\\
& q_t+ q\cdot \nabla q + \rho^{-1} p_\rho \nabla \rho+
\rho^{-1} p_e \nabla e=0,\\
& e_t+ q\cdot \nabla e + \rho^{-1} p \div u=0,\\
\end{split}
\end{align}
where $q=(u,v)$ denotes velocity, or, in one dimension,
$$
V_t + (u Id + M)V_{x_1}=0,
$$
where $V=(\rho, u,v,e)$ and
$$
M:=
\begin{pmatrix}
0 & \rho & 0 & 0\\
\rho^{-1}p_\rho & 0 & 0 & \rho^{-1}p_e\\
0 & 0 & 0 & 0\\
0 & \rho^{-1}p & 0 & 0\\
\end{pmatrix},
$$
from which we may conclude that
$ A_1= S(u Id + M)S^{-1} $ for
$$
S:=
\frac{\partial \big(\rho, \rho u, \rho v, \rho(e+u^2/2+v^2/2)\big)}
{\partial (\rho, u, v, e)}
=
\begin{pmatrix}
1 & 0& 0 & 0\\
u & \rho & 0 & 0\\
v & 0 &  \rho & 0 \\
e+ u^2/2 + v^2/2 & \rho u & \rho v & \rho\\
\end{pmatrix},
$$
$$
S^{-1}:=
\begin{pmatrix}
1 & 0& 0 & 0\\
\frac{-u}{\rho} & \frac{1}{\rho} & 0 & 0\\
\frac{-v}{\rho} & 0 &  \frac{1}{\rho} & 0 \\
\frac{-e+ u^2/2 + v^2/2}{\rho} & \frac{-u}{\rho}  & \frac{-v}{\rho} & \frac{1}{\rho}\\
\end{pmatrix},
$$
and thus
$\ell_+^*=\tilde \ell_+^* S^{-1}$ for $\tilde \ell_+$ defined as the
left eigenvector of $M$ associated with the eigenvalue of smallest real part,
$*$ denoting adjoint, or congugate transpose, all quantities to
be evaluated at $(\rho, u,v, e)=(1/u_+, u_+, 0, e_+)$.

By inspection, $\tilde \ell_+^*=(p_\rho, -\rho c, 0, p_e)$ for
$v=0$, where sound speed $c$ is defined as in \eqref{c},
whence 
$$
\ell_+^*=
\tilde \ell_+^* S^{-1}=
\Big(p_\rho + cu + \frac{p_e(u^2/2-e)}{\rho},  -\frac{p_eu}{\rho}-c, 
0, \frac{p_e}{\rho}\Big) (U_+)
$$
as claimed.

\end{document}